\documentclass[11pt]{amsart}
\usepackage[english]{babel}
\usepackage{amsfonts}
\usepackage[utf8]{inputenc}
\usepackage{amsthm}
\usepackage{ mathrsfs }
\usepackage{ amsmath }
\usepackage{mathrsfs}
\usepackage{enumerate}
\usepackage{cite}
\usepackage{bbm}
\usepackage{graphicx}
\usepackage{frontespizio}
\usepackage{xcolor}
\usepackage[makeroom]{cancel}
\usepackage[normalem]{ulem}
\usepackage{ amssymb }
\usepackage{hyperref}
\usepackage{mathtools}
\usepackage{enumitem}
\usepackage[toc,page]{appendix}
\usepackage{ esint }
\usepackage[margin=2.5cm]{geometry}

\newcommand{\eps}{\varepsilon}
\newcommand{\BV}{{\sf BV}}
\newcommand{\rr}{\mathbb{R}}
\newcommand{\nn}{\mathbb{N}}

\newcommand{\nchi}{{\raise.3ex\hbox{$\chi$}}}

\newcommand{\sfd}{{\sf d}}
\newcommand{\Lip}{{\rm Lip}}

\renewcommand{\phi}{\varphi}
\newcommand{\restr}[1]{\lower3pt\hbox{$|_{#1}$}}

\newcommand{\X}{{\rm X}}

\newcommand{\fr}{\penalty-20\null\hfill$\blacksquare$} 
\definecolor{mygray}{gray}{0.9}
\newcommand{\diam}{\text{diam}}

\newcommand{\la}{\langle}
\newcommand{\ra}{\rangle}
\renewcommand{\div}{{\rm div}}
\newcommand{\mea}{\mathfrak{m}}
\newcommand{\mm}{\mathfrak{m}}
\newcommand{\Per}{{\rm{Per}}}
\newcommand{\LIP}{\mathsf{LIP}}

\newcommand{\test}{{\sf{Test}}}
\newcommand{\bd}{{\mbox{\boldmath$\Delta$}}}

\renewcommand{\Cap}{{\rm Cap}}

\renewcommand{\d}{{\mathrm d}}
\newcommand{\loc}{\mathsf{loc}}
\newcommand{\W}{\mathit{W}^{1,2}}

\newcommand{\supp}{\mathop{\rm supp}\nolimits} 

\newcommand{\D}{{\sf D}}
 
\newcommand{\e}{{\rm{e}}}    
\newcommand{\lip}{{\rm lip}}

\DeclareMathOperator*{\esssup}{ess\,sup}
\DeclareMathOperator*{\essinf}{ess\,inf}

\renewcommand{\c}{{\mathrm c}}
\newcommand{\Xdm}{(\X,\sfd,\mm)}
\newcommand{\RCD}{\mathrm{RCD}}
\newcommand{\CD}{\mathrm{CD}}
\newcommand{\cd}{\mathrm{CD}}
\newcommand{\R}{\mathbb{R}}

\newcommand{\testv}{{\rm TestV}}

\newcommand{\rcd}{\mathrm{RCD}}

\newcommand{\wil}{{\mathcal{W}}}
\newcommand{\Dwil}{{\sf D}({\mathcal{W}})}

\newcommand{\sgn}{{\sf sgn}}

\newcommand{\dom}{{\sf D}}
\newcommand{\pert}{\Per(\{u<t\})}

\renewcommand{\Cap}{{\rm Cap}}

\newcommand{\cal}{\mathcal}
\renewcommand{\limsup}{\varlimsup}
\renewcommand{\liminf}{\varliminf}
\newcommand{\avr}{{\sf AVR}}

\theoremstyle{plain}
\newtheorem{theorem}{Theorem}[section]
\newtheorem{lemma}[theorem]{Lemma}
\newtheorem{prop}[theorem]{Proposition}

\newtheorem{cor}[theorem]{Corollary}

\theoremstyle{definition}
\newtheorem{definition}[theorem]{Definition}
\newtheorem{remark}[theorem]{Remark}

\newcommand{\TD}{{\rm TD}}

\newcommand{\Int}{\mathcal I_{TV}}
\newcommand{\Hu}{\overrightarrow{H}}

\numberwithin{equation}{section}

\begin{document}
\begin{abstract}
The goal of this work is to introduce a notion of mean curvature for level sets of functions in non-smooth spaces with Ricci curvature bounded below, and to prove that it satisfies sharp geometric inequalities. More precisely, we define a suitable Willmore functional $\mathcal{W}$ on Sobolev functions, whose domain of finiteness is dense in $L^p$ for any $1\le p<\infty$. For any function with finite Willmore energy, we show that almost all of its level sets admit a mean curvature vector satisfying the natural integration by parts formula with respect to the tangential divergence. As a main application, we show that in $\RCD(0,N)$ spaces with Euclidean volume growth, almost every level set of the electrostatic potential possesses a mean curvature vector in the above sense. Furthermore, we prove that this vector satisfies the same sharp Willmore inequality as in the smooth setting, alongside rigidity and almost-rigidity statements. Finally, as a  technical tool, we generalize the sharp isocapacitary inequality to the non-smooth setting.

\end{abstract}

	\title[Mean curvature in metric spaces]{Mean curvature and sharp Willmore inequalities  \\ in  metric spaces}

	\author[N.~Gigli]{Nicola Gigli}
	\address[N.~Gigli]{Scuola Internazionale Superiore di Studi Avanzati (SISSA),  Via Bonomea 265,  34136, Trieste, Italy}\email{ngigli@sissa.it}
	
	\author[I.~Y.~Violo]{Ivan Yuri Violo}
	\address[I.~Y.~Violo]{Universit\`a di Pisa, Dipartimento di Matematica, Largo Bruno Pontecorvo, 5, 56127 Pisa, Italy}
	\email{ivanyuri.violo@dm.unipi.it }
	
\maketitle
    
	\tableofcontents

\section{Introduction}

The  goal of this paper is to understand whether a meaningful notion of mean curvature exists within the framework of metric measure spaces satisfying the curvature-dimension condition. In recent years, a robust theory of first- and second-order calculus has been successfully developed on these spaces (see e.g.\ \cite{GigliPasqualetto20book,Gigli14,Bjorn-Bjorn11,Heinonen_Koskela_Shanmugalingam_Tyson_2015} and references therein) where objects are typically defined almost everywhere with respect to the reference measure. Regarding codimension-one objects, there is also a well-established and satisfactory theory for perimeter and isoperimetric sets  (see e.g.\ \cite{amb00,Miranda03,brue2023rectifiability,APPV23}). 

However, understanding second-order objects localized on codimension-one sets, such as mean curvature, remains a significant challenge. A robust notion of mean curvature is needed in connection with geometric measure theory, minimal surfaces, isoperimetric sets, and functional inequalities. 
In the existing literature, there have already appeared notions of weak bounds for the mean curvature in non-smooth spaces with Ricci curvature bounded below relying primarily on 1D localization techniques \cite{ketterer2020heintzekarcher,cavalletti2024optimal} or weak Laplacian bounds \cite{mondino2025weak}. Our goal is to provide a pointwise definition of the mean curvature vector that is more directly linked to the framework of second-order calculus.

Our construction does not provide a notion of mean curvature for arbitrary sets. Instead, it applies to almost every level set of a rich class of sufficiently regular functions, which effectively fiber the space. 

Our strategy to bypass the lack of a priori smoothness of these level sets relies on a duality argument. We first  introduce, for a class of `smooth enough' vector fields,  a weak notion of tangential divergence with respect to the level sets of a given function $u \in \W(\X)$. Specifically, we define the tangential divergence as
\begin{equation}
    \TD(v)\coloneqq \div (v)-\nabla v \left(\frac{\d u}{|\d u|},\frac{\d u}{|\d u|}\right),
\end{equation}
which is well-defined $\mm$-a.e.\ on the regular set $\{|\d u|>0\}$, see Definition \ref{def:tandiv}. Note that this is precisely the usual tangential divergence in the smooth setting and corresponds to removing from the divergence the normal component of the covariant derivative. 

With this tool at hand, we can introduce the Willmore functional, denoted by $\wil(u)$. Formally, this functional represents half the integrated $L^2$-norm of the mean curvature across all level sets of $u$. Since we do not yet have a notion of mean curvature, we define $\wil(u)$ rigorously  by formal integration by parts against the tangential divergence:
\begin{equation}
    \wil(u)\coloneqq \sup_{v \in \testv(\X)} \int_{\X} \left(\TD(v)-\frac{|v|^2}{2}\right) |\d u| \d\mm.
\end{equation}
To see formally how this formula relates to the norm of the mean curvature on level sets  see the beginning of Section \ref{sec:willmore functional}. 
Related  operators, involving the integration of the $L^p$-norm of the mean curvature over the level sets of smooth functions, have been previously studied in the Euclidean setting particularly in  image processing; see, e.g. \cite{MasnouNardi+2013+433+482,MR3136594,AmbrosioMasnou2003}.

Assuming the finiteness of $\wil$, we can exploit the Riesz representation theorem to extract an $L^2$-integrable vector field that acts precisely as the mean curvature. This leads to our main existence result:
\begin{theorem}\label{thm:main H intro}
Let $\Xdm$ be an $\RCD(K,\infty)$ space. For all $u \in \W(\X)$ such that $\wil(u)<\infty$ there exists a $|\d u|\mm$-a.e.\ unique vector field $\Hu(u) \in L^0(T\X)$ with $|\Hu(u)|\in L^2(|\d u|\mm)$, called mean curvature vector, such that for all $v\in \testv(\X)$
\begin{equation}\label{eq:mean curv eq on level sets intro}
\int \la \Hu(u),v\ra \,\d \pert =\int \TD(v)\,\d \pert, \quad \text{for a.e.\ $t \in \rr$.}
\end{equation}
Moreover
\begin{equation}\label{eq:wil int fomrula intro}
    \wil(u)=\frac12\int_\rr \int |\Hu(u)|^2\d \pert\, \d t.
\end{equation}
   Finally, if $u\in \dom(\Delta)\cap \LIP(\X)$, then $\wil(u)<\infty$ if and only if 
 \begin{equation}\label{eq:H formula intro}
     H(u)\coloneqq \left(\frac{\Delta u}{|\d u|} -\frac{\la \nabla |\d u|,\nabla u \ra }{|\d u|^2}\right)\in L^2(|\d u|\mm),
 \end{equation}
in which case $\Hu (u)=H(u)\nchi_{|\d u|>0}\frac{\d u}{|\d u|}$.
\end{theorem}
The notion of mean curvature given in the above statement is meaningful because it satisfies the expected integration by parts formula \eqref{eq:mean curv eq on level sets intro} with respect to the tangential divergence directly on the level sets. In particular it coincides with the classical notion of distributional mean curvature in the smooth setting for regular level sets, whenever it is also  $L^2$-integrable on that level set.

Identity \eqref{eq:wil int fomrula intro}, on the other hand, says that  a posteriori the Willmore functional indeed coincides with half the total $L^2$-norm of the mean curvature on the level sets of $u.$ 

Finally, the last part of the statement provides the link with the second order calculus: if the function $u$ is smooth enough we deduce that  our notion of mean curvature coincides with the usual formula in the smooth setting. Indeed, note that the right-hand side of \eqref{eq:H formula intro} corresponds formally to $\div(|\nabla u|^{-1}\nabla u).$   Recall that $\dom(\Delta)$ denotes the space of $\W(\X)$-functions with Laplacian in $L^2(\mm).$

\begin{remark}[Existence of many functions having finite Willmore functional]
    It is a priori not obvious whether there exist non-trivial functions $u\in W^{1,2}(\X)$ with $\wil(u)<\infty$, not even in the subclass $u\in \dom(\Delta)\cap \LIP(\X)$.  Crucially, we will show that in $\RCD(K,N)$ spaces with $N<\infty$ such functions constitute a rich class. In particular, we prove that the domain of finiteness of the Willmore functional $\wil$ is dense in $L^p(\mm)$ for all $p\in[1,\infty)$; see Theorem \ref{thm:domain dense} for the precise statement. \fr
\end{remark}

\begin{remark}[Comparison with previous notions of mean curvature]

Several notions of mean curvature in non-smooth metric spaces have already appeared in the literature. While these approaches share some of our motivation, they differ substantially from ours, with complementary strengths and limitations.

For instance, \cite{ketterer2020heintzekarcher} and \cite{cavalletti2024optimal} (see also \cite[Appendix B]{kettererKasue}) introduced lower and upper bounds for the mean curvature of a set using 1D-localization techniques in non-smooth spaces with curvature bounds. In the $\RCD$ setting, \cite{mondino2025weak} showed that locally perimeter-minimizing sets have vanishing mean curvature in a weak sense by means of bounds on the Laplacian of the distance function. See also \cite{PanuMalyShanSpeight} for weak mean curvature bounds in PI spaces. 

The defining feature of the present work is that we provide an explicit definition of the mean curvature vector itself, whereas previous approaches mainly establish scalar upper and lower bounds. (The only partial exception is \cite{ketterer2020heintzekarcher}, where a pointwise notion is introduced, but only for the inner and outer mean curvatures.)

The main limitation of our approach is that it applies only to almost every level set of a rich class of sufficiently regular functions, rather than to an arbitrary prescribed set. By contrast, the above-mentioned works apply directly to individual sets, although typically under additional assumptions such as inner or outer ball conditions, perimeter minimization, or regularity of the surface measure.
 
 Another distinctive feature of our definition is its direct connection with second-order calculus. Indeed, the mean curvature vector is characterized by the integration-by-parts identity with respect to the tangential divergence \eqref{eq:mean curv eq on level sets intro} and, whenever the underlying function is smooth enough, it admits the explicit representation \eqref{eq:H formula intro}. This analytic structure is largely absent from previous distributional approaches and is the key ingredient allowing us to derive the sharp geometric inequalities of Theorem \ref{thm:sharp willmore}, which would be difficult to obtain using previous frameworks (compare with the different inequalities established in \cite{ketterer2020heintzekarcher} and \cite{kettererKasue}).
\fr
 
\end{remark}

Since Theorem \ref{thm:main H intro} is formulated for globally defined functions, we also prove a local version of it for functions defined only in an open set $\Omega \subset \X$, with the only requirement of having level sets relatively compact in $\Omega$. The statement is essentially identical and can be obtained with the same argument localized, see  Proposition \ref{prop:main H local}.

Thanks to the local version of the existence result for the mean curvature vector, we are able  to show that \textit{harmonic functions} with relatively compact level sets admit  an $L^2$-mean curvature  for almost every level set in the above sense, see Corollary \ref{cor:H formula for harmonic functions} for the precise version. This fact not only provides an important and explicit large class of examples of functions to which our theory applies in a non-trivial way, but  also has important geometric consequences as we shall see below.

A particularly relevant class of examples satisfying these assumptions is provided by \textit{electrostatic potentials}. These are  harmonic functions defined on some outer domain $\X\setminus K$, equal to one in $K$ and  vanishing at infinity,  see Definition \ref{def:electro} for the precise notion in our setting.

Electrostatic potentials  play an important role in recent geometric analysis because their level sets naturally foliate the exterior domain, providing a powerful alternative to geometric flows. It was discovered in \cite{Co12}  that on Riemannian manifolds with non-negative Ricci curvature there are important monotone quantities along the level sets of the Green function of the Laplacian. Later in \cite{AgostinianiFogagnoloMazzieri20} this was generalized to the case of electrostatic potentials. More importantly for us, in \cite{AgostinianiFogagnoloMazzieri20} these monotonicity properties were used to obtain sharp Willmore-type inequalities for the mean curvature of hypersurfaces (see also the previous \cite{agostiniani2020monotonicity}  in the Euclidean setting). For more recent advances and generalizations of this approach see \cite{benatti2024minkowski,benatti2024fine,FOGAGNOLOpina} and references therein.

These same monotonicity formulas for the electrostatic potential were generalized to the setting of metric measure spaces with non-negative Ricci curvature by the authors in \cite{GigliViolo23}. Combining these results with the notion of mean curvature developed in this note and described above, we are finally able to deduce sharp Willmore inequalities in this more general setting.

\begin{theorem}[Sharp Willmore-type inequality]\label{thm:sharp willmore}
Let $\X$ be an $\rcd(0,N)$ space, $N>2$, with Euclidean volume growth, and let $u$ be an electrostatic potential for a compact set $K$. Let $\Hu(u)$ be the mean curvature given in Proposition \ref{prop:main H local} and for all $\beta \in \big[\frac{N-2}{N-1},1\big]$ define the Willmore energy 
\[
(0,1)\ni t\mapsto W_{\beta+1}(t)\coloneqq \int |\Hu (u)|^{1+\beta} \d \pert.
\]
Then $W_{\beta+1}(t)$  admits a (maximal) lower semicontinuous representative in $(0,1)$. Moreover, for this choice of representative it holds
\begin{equation}\label{eq:willmore ineq intro}
\frac{W_{\beta+1}(t)}{(N-1)^{\beta+1}} \ge(\avr(\X)\sigma_N)^{\frac{\beta}{N-2}}t^{\frac{\beta}{N-2}-1}\left(\frac{{\rm Cap}(K)}{N-2}\right)^{1-\frac{\beta}{N-2}}, \quad \text{for all $t \in (0,1).$}
\end{equation}
Finally, if equality holds for some $t_0\in(0,1)$ and $\beta>\frac{N-2}{N-1}$, then $\{u<t_0\}$ is locally isometric to $Y\setminus B_R(O_Y)$, where $Y$ is an $N$-Euclidean cone with tip $O_Y$ over an $\RCD(N-2,N-1)$ space.
\end{theorem}
Note that the first part of the result gives a mild but non-trivial extra regularity  for the Willmore energy along the level sets. Indeed, even though the vector $\Hu(u)$ is defined a priori only $\mm$-a.e.\ and thus $W_{\beta+1}(t)$ is well defined only up to a.e.-equivalence classes, $W_{\beta+1}$ admits a natural pointwise defined value.

An almost-rigidity result for the Willmore inequality is also established, which we believe to be new even in the smooth setting, stating that if almost equality is achieved in \eqref{eq:willmore ineq intro}, then the space is close to a truncated cone in the pointed measure Gromov-Hausdorff (pmGH) sense, see Theorem \ref{thm:almost rigidity} for the rigorous statement.

As an intermediate tool for inequality \eqref{eq:willmore ineq intro} we will also obtain the sharp isocapacitary inequality for $\RCD(0,N)$ spaces with Euclidean volume growth, stating that
		\begin{equation}\label{eq:isocap intro}
			\Cap(E)\ge(N-2)N \omega_N^{\frac2N} \avr(\X)^{\frac2N} \mea(E)^\frac{{N-2}}{N},
            \end{equation}
            for every bounded Borel set $E\subset \X.$ This inequality generalizes the one  for $p=2$ in \cite[Section 4.1]{benatti2024minkowski} for smooth manifolds and  we believe it to be of independent interest.  We also prove that equality in \eqref{eq:isocap intro} is achieved if and only if  $E$ is $\mm$-a.e.\ equal to a ball and  $\X$ is a cone, see Theorem \ref{thm:isocap} for the precise statement.

Alongside the geometric inequalities discussed above, a further crucial contribution of this paper concerns the stability of the Willmore functional. Specifically, we establish its lower semicontinuity under pmGH-convergence.
\begin{theorem}[$\Gamma-\liminf$ of the Willmore functional]\label{thm:lsc}
Let $\X_n$, $n \in \nn\cup\{\infty\}$ be a pmGH-converging sequence of $\RCD(K,\infty)$ spaces, let $\wil_n$ denote the Willmore functional in $\X_n$. Then, for any $\W$-converging sequence  $u_n\to u_\infty$ it holds
\[
\wil_\infty(u_\infty)\le \liminf_n \wil_n(u_n).
\]
\end{theorem}

This paper is organized as follows. Section \ref{sec:pre} recalls the necessary preliminaries. Section \ref{sec:tandiv and wil} introduces the Willmore functional and the mean curvature vector. Section \ref{sec:lsc} establishes the lower semicontinuity of the Willmore functional, while Section \ref{sec:reg along level sets} studies the regularity of the Willmore energy on level sets of harmonic functions. Finally, Section \ref{sec:isocap will} is devoted to the proofs of the sharp isocapacitary and Willmore inequalities, including rigidity and almost-rigidity results.

	\section{Preliminaries}\label{sec:pre}
	
Throughout this paper, $\Xdm$ will denote a metric measure space. We will denote by $\LIP_{bs}(\X) $ the space of Lipschitz functions in $\X$ with bounded support. For a given open set $\Omega\subset \X$ we denote by $\LIP(\Omega), \LIP_\loc(\Omega)$ and $\LIP_c(\Omega)$ respectively the spaces of Lipschitz functions in $\Omega$, locally Lipschitz functions in $\Omega$ and Lipschitz functions with compact support in $\Omega$.  We assume the reader to be familiar with the foundational theory of calculus on metric measure spaces. In particular, we take for granted the definition and properties of the Sobolev spaces $\W(\Omega) $, $\W_\loc(\Omega)$, of functions of (locally) bounded variation  $\BV(\X), \BV_\loc(\X)$, the first order differential structure including tangent and cotangent modules and the theory of normed modules. We also assume the reader to be familiar with  the definition of $\RCD$ spaces and  second order calculus  including the notions of covariant derivative and tensor calculus. For a comprehensive overview and detailed proofs concerning these topics, we refer the reader to the existing literature (see, e.g., \cite{Gigli14, GigliPasqualetto20book, Bjorn-Bjorn11,AmbrosioDiMarino14,Gigli23_working}).

In this section, we only briefly recall specific notations, definitions, and a few preliminary results that are tailored to our exposition and  necessary for the upcoming proofs.

We start with the notion of capacity.
	\begin{definition}[Capacity]\label{def:capacity}
		Let $\Xdm$ be a metric measure space and  $E\subset \X$ be a Borel subset. We define $\Cap(E)\in [0,\infty]$ as
		\[
		\Cap(E)=\inf\, \left \{ \int_\X |\d u|^2\, \d \mea \ : \ u \in \W(\X) \text{ and $u\ge 1$ $\mm$-a.e.\  in a neighborhood of $E$}\right \}.
		\]
	\end{definition}
     The function $\Cap(\cdot)$ is an outer measure and any function $u \in \W_\loc(\X)$ has a $\Cap$-a.e.\ unique quasi-continuous representative  denoted by ${\sf QCR}(u)$ (see \cite{debin2021quasi}). Recall that a function is said to be quasi-continuous if it is continuous outside a set of arbitrarily small capacity.

	\begin{lemma}[Equivalent definitions of capacity]\label{lem:cap bs}
		Let $E\subset \X$ be a bounded Borel set. Then the infimum in $\Cap(E)$ can be taken equivalently among all functions in any one of the following families:
        \begin{enumerate}[label=\roman*)]
        \item $u \in \W(\X)$ non-negative such that $u\ge 1$ $\mm$-a.e.\ in a neighborhood of $E$ (or ${\sf QCR}(u)\ge 1$ $\Cap$-a.e.\ in $E$).
            \item $u \in \W(\X)$ non-negative with bounded support such that $u\ge 1$ $\mm$-a.e.\ in a neighborhood of $E$ (or ${\sf QCR}(u)\ge 1$ $\Cap$-a.e.\ in $E$);
            \item   $u \in \W_\loc(\X)$ non-negative and vanishing at infinity\footnote{By vanishing at infinity for an $\mm$-a.e.\ defined function we mean that  $\lim_{R\to +\infty }\esssup_{\X\setminus B_R(x_0)} |u|=0.$}  such that $u\ge 1$ $\mm$-a.e.\ in a neighborhood of $E$ (or  ${\sf QCR}(u)\ge 1$ $\Cap$-a.e.\ in $E$).
        \end{enumerate}
	\end{lemma}
    \begin{proof}
    It is sufficient to show the result in the cases when $u\ge 1$ $\mm$-a.e.\ in a neighborhood of $E$, then the version where we replace it with ${\sf QCR}(u)\ge 1$ $\Cap$-a.e.\ in $E$ follows by standard arguments (see \cite[Theorem 5.31]{Bjorn-Bjorn11}). 

    i) Immediate from the definition of $\Cap(E),$ since taking the positive part decreases the energy.
    
       ii) Fix $x\in \X$ and fix any $R>0$ big enough so that $E\subset B_R(x).$ Take $\phi_R\in \LIP_{bs}(\X)$ such that, $0\le \phi_R\le 1$, $\phi_R=1$ in $B_{R}(x)$, $\phi_R=0$ in $\X\setminus B_{2R}(x)$ and $\Lip(\phi_R)\le 1/R$. For any $\delta>0$ there exists $u\in \W(\X)$ non-negative with  $u\ge 1$ $\mm$-a.e.\ in a neighborhood of $E$ and $\int_\X |\d u|^2\, \d \mea\le \Cap(E)+\delta$. Then $\phi_Ru\in \W(\X)$ has bounded support, $\phi_Ru\ge 1$  $\mm$-a.e.\ in a neighborhood of $E$  and
        \[
        \| |\d (u\phi_R)|\|_{L^2(\mm)}\le  \| |\d u|\|_{L^2(\mm)} + R^{-1}\|u\|_{L^2(\mm)}\le  \sqrt{\Cap(E)+\delta}+ R^{-1}\|u\|_{L^2(\mm)}.
        \]
        Sending $R\to +\infty$ and by the arbitrariness of $\delta$ we conclude.

        Item iii)  follows noting that for any $u$ as in the family and for all $\delta\in(0,1)$ it holds that $u_\delta \coloneqq (1-\delta)^{-1}(u-\delta)^+\in \W(\X)$ and that $u_\delta\ge 1$    $\mm$-a.e.\ in a neighborhood of $E$ and then taking $\delta \to 0^+.$
    \end{proof}

Next we recall the definition of Laplacian following \cite[Section 4]{Gigli12} (see also \cite{GigliPasqualetto20book}).
\begin{definition}[Laplacian]
    Let $\Xdm$ be an infinitesimally Hilbertian\footnote{$\Xdm$ is said to be inf.\ Hilbertian if $\W(\X)$ is a Hilbert space \cite{Gigli12}} m.m.s.\ and $\Omega\subset \X$ be open. We say that  $u \in \W_\loc(\Omega) $ belongs to the domain of the \emph{Laplacian} $\dom(\Delta,\Omega)$ if there exists  $h \in L^2_{\loc}(\Omega)$ such that
    \begin{equation}
        -\int_{\Omega} \langle \nabla f, \nabla u\rangle \d \mea=\int_{\Omega} f h \,\d \mm, \quad 	\text{for every $f \in \LIP_c(\Omega).$}
    \end{equation}
    Moreover the unique function $h$  with this property is called the \textit{Laplacian of $u$} in $\Omega$ and is denoted by $\Delta\restr{\Omega} u$.
\end{definition}

As shorthand we will often omit writing $u \in \dom(\Delta,\Omega)$ and write directly $\Delta\restr{\Omega} u = h$ for some function $h\in L^2_\loc(\Omega)$ or $\Delta\restr{\Omega} u \in L^2_\loc(\Omega).$ We will also write as usual $\dom(\Delta)\coloneqq \{u \in \dom(\Delta,\X)\cap \W(\X) \ : \ \Delta u \in L^2(\mm)\}$, see e.g.\  \cite{Gigli14,GigliPasqualetto20book}.
% [Qui puoi in

We say that a function $u \in \W_\loc(\Omega)$ is \textit{harmonic} if $\Delta\restr{\Omega} u = 0$. It is known that harmonic functions are continuous in any $\RCD(K,N)$ space with $N<\infty$ and in fact locally Lipschitz (see \cite{Jiang13,VIOLO2025}).

	\begin{definition}[Electrostatic potential]\label{def:electro}
		Let $\Xdm$ be a metric measure space and $K\subset \X$ be compact. An electrostatic potential for $K$ is a function $u\in C(\X\setminus K)$ that solves
		\[
		\begin{cases}
			u(x)\to 1, &\text{ as $x\to \partial K$ },\\
			u(x)\to 0,& \text{ as $\sfd(x,K)\to +\infty$ },\\
			\Delta\restr{\X\setminus K} u=0.
		\end{cases}
		\]
	\end{definition}

	The following result links the electrostatic potential and the capacity.
	\begin{prop}\label{prop:electro cap formula}
		Let $\Xdm$ be an $\RCD(0,N)$ space  with $N>2$ and\footnote{$\avr(\X)\coloneqq \lim_{r\to +\infty}\frac{\mm(B_r(x_0))}{\omega_Nr^N}$, which is independent of $x_0.$} $\avr(\X)>0$. Let $E\subset \X$ be a bounded Borel set. Then there exists $u\in \W_\loc(\X)\cap C(\X\setminus\bar E)$, harmonic in $\X\setminus \bar E$  such that $0\le u\le 1$, $u\ge 1$ $\mm$-a.e.\ in $E$,  $u(x)\to 0$ as $\sfd(x,E)\to +\infty$ and 
		\begin{equation}\label{eq:electro cap formula}
			\Cap(E)=\int_{\X}|\d u|^2 \d \mm.
		\end{equation}
        Moreover if $E$ is compact and admits an electrostatic potential $\bar u$, we have $u=\bar u$ in $\X\setminus E.$
	\end{prop}
	\begin{proof}
    The proof is very similar to \cite[Theorem 8.4]{GigliViolo23}.  Let $\Cap(E,B_n(x_0))$ denote the relative capacity, i.e.\ the infimum of $\int |\d v|^2\d \mm$ among all $v\in \W_0(B_n(x_0))$ with $v\ge 1$  $\mm$-a.e.\ in a neighborhood of $E$. As in Lemma \ref{lem:cap bs}, $\Cap(E,B_n(x_0))$ is also realized by taking the infimum among all $v\in \W_0(B_n(x_0))$ with ${\sf QCR}(v)\ge 1$  $\Cap$-a.e.\ in  $E$. Since this set of competitors is convex and weakly closed in $\W(\X)$ there exists a minimizer $u_n$ with $0\le u_n\le 1$. Moreover $u_n\le v$ $\mm$-a.e.\ in $B_n(x_0)$ for any $v\in \W(B_n(x_0))$ superharmonic and such that $v\ge 1$ $\Cap$-a.e.\ in $E$ (the proof of this is the same as in \cite[Lemma B.9]{GigliViolo23} replacing $\mm$-a.e.\ with $\Cap$-a.e.).  In particular $u_n\le G$ $\mm$-a.e.\ in $B_n(x_0)$ for a function $G$ independent of $n$ and vanishing at infinity (see the  end of the proof of \cite[Theorem 8.4]{GigliViolo23}). Additionally $u_n$ is harmonic in $B_n(x_0)\setminus \bar E.$ It is clear that  $\Cap(E,B_n(x_0))$ is non-increasing and by Lemma \ref{lem:cap bs}
\[
\lim_{n\to +\infty }\Cap(E,B_n(x_0))=\Cap(E).
\]
    Since $(u_n)_n$ is bounded in $\W(B_R(x_0))$ for all $R>0$ and each $u_n$ harmonic in $B_n(x_0)\setminus \bar E$, up to a subsequence, $u_n$  converges weakly in $\W_\loc(\X)$ and locally uniformly in $\X\setminus \bar E$ to some $u\le G$. By lower semicontinuity $\int_{\X} |\d u|^2\d \mm\le \liminf_n \int_{\X} |\d u_n|^2\d \mm=\Cap(E).$  In particular $u\in C(\X\setminus\bar E)$ and vanishes at infinity and  ${\sf QCR}(u)\ge 1$ $\Cap$-a.e.\ in $E$. Hence by Lemma \ref{lem:cap bs} we conclude $\int_{\X} |\d u|^2\d \mm=\Cap(E).$
    
    Suppose now that $E=K$ is compact and admits an electrostatic potential $\bar u.$
		For any $\eps>0$ set $K^\eps$ to be the $\eps$-open tubular neighborhood of $K$. From \cite[Theorem 8.4]{GigliViolo23}, which can be applied because $K^\eps$ has cap-fat boundary (see the end of page 80 in \cite{GigliViolo23}),  we know that the electrostatic potential $u_\eps$ of $\overline {K^\eps}$ exists  and 
		$$\int_{\X\setminus \overline {K^\eps} } |\d u_\eps|^2\d \mm\le \lim_{R \to +\infty} \Cap(K^\eps,B_R(x_0)),$$
		for any choice $x_0\in\ X.$ We also have $u_\eps \in \W_\loc(\X)$ by setting $u_\eps=1$ in $K^\eps.$  Take any $\delta>0.$ By Lemma \ref{lem:cap bs} there exists $v\in \W(\X)$ with bounded support  such that $v\ge 1$  $\mm$-a.e.\ in a neighborhood of $K$ such that $\int |\d v|^2\d \mm \le \Cap(K)+\delta $. Hence  $v\in \W_0(B_R(x_0))$ for any $R$ big enough and  $v\ge 1$ $\mm$-a.e.\ in $K^\eps$ for any $\eps$ small enough  (independently of $R$) and thus  $v$  is admissible for $\Cap(K^\eps,B_R(x_0))$. We obtain
		\[
		\liminf_{\eps \to 0^+} \lim_{R \to +\infty} \Cap(K^\eps,B_R(x_0))\le \Cap(K)+\delta .
		\]
		Since $\bar u(x)\to 1$ as $x\to \partial K$ we have that $\sup_{\partial K^\eps }| \bar u-1|\to 0$ as $\eps \to 0.$ On the other hand $u_\eps=1$ in $K_\eps$. Moreover both $\bar u$ and $u_\eps$ vanish at infinity. Therefore by the comparison principle
		\[
		\sup_{\X\setminus K^\eps }|\bar u-u_\eps| \le \sup_{\partial K^\eps }| \bar u-1|\to 0.
		\]
		Therefore by lower semicontinuity of the energy 
		\begin{align*}
			\int_{U} |\d \bar u|^2\d \mm&\le \liminf_{\eps \to 0^+} \int_{U} |\d u_\eps|^2\d \mm\le  \liminf_{\eps \to 0^+}  \int_{\X\setminus  \overline {K^\eps} } |\d u_\eps|^2\d \mm\\
			&\le 
			\liminf_{\eps \to 0^+}  \lim_{R \to +\infty} \Cap(K^\eps,B_R(x_0))\le \Cap(K)+\delta , 
		\end{align*}
		for all $U\subset \subset \X\setminus K.$ By the arbitrariness of $U$ and $\delta$ we obtain 
		\[
		\int_{\X\setminus K}|\d \bar u|^2 \d \mm\le  \Cap(K).
		\]
        Moreover $\bar u \in \W_\loc(\X)$ by setting $\bar u=1$ in $K$, because $u_\eps \in \W_\loc(\X)$ and again by lower semicontinuity.
        Hence by Lemma \ref{lem:cap bs} we conclude $\int_{\X} |\d \bar u|^2\d \mm=\Cap(K).$ The fact that $\bar u=u$ in $\X\setminus K$ follows from the fact that the function $\frac{u+\bar u}{2}$ is an admissible competitor for $\Cap (K)$ (see Lemma \ref{lem:cap bs}) combined with strict convexity.
	\end{proof}

The following regularity result for harmonic functions was shown in \cite[Section 4]{GigliViolo23}.
    \begin{theorem}\label{cor:harmest2}
		Let $(\X,\sfd,\mea)$ be an ${\sf RCD}(K,N)$ space, $N \in[2,\infty)$, and let $u$ be harmonic in $\Omega\subset \X$ open.   Then for all $\beta \ge \frac{N-2}{N-1}$ it holds $|\d u|^{\frac \beta2} \in \W_\loc(\Omega)$  and $|\d u|^\beta\in \dom(\bd,\Omega)$. In particular
		\begin{equation}\label{eq:harmest2}
			\nchi_{\{|\d u|>0\}}\frac{|\d |\d u||^2}{|\d u|} \in L^1_\loc(\Omega).
		\end{equation}
	\end{theorem}
	
	Recall the class of test functions $\test(\X)\coloneqq \left\{f\in L^\infty\cap \LIP(\X)\cap \dom(\Delta)\ : \  \Delta f\in \W\cap L^\infty(\X)  \right\}$ and test vector fields   $\testv(\X)\coloneqq\{\sum_{i=1}^m f_i\nabla g_i \ : \ f_i,g_i\in \test(\X) \}$   (see \cite{Gigli14,Savare13}).

    In the sequel we will use the following approximation lemma which we could not find in the literature.
	\begin{lemma}\label{lem:approximation with bounded testv}
		Let $\Xdm$ be an $\RCD({ K},\infty)$ space and $v\in L^2(T\X)$ with $|v|\le 1$. Then for all $\delta>0$ there exists $v_n \in \testv(\X)$ with $|v_n|\le 1+\delta$ and such that $v_n \to v$ in $L^2(T\X)$.
	\end{lemma}
	\begin{proof}
        \noindent {\bf Step 1}:  \textit{$v=\sum_{i=1}^m\nchi_{B_i} \nabla f_i$ with $f_i\in \test(\X)$,   $\Lip(f_i)\le 1$ and $B_i\subset \X$ disjoint compact sets.} 

        For all $n\in\nn$ and for all $i=1,\dots,m$ set $\alpha_{i,n}\coloneqq H_{1/n}(\nchi_{B_i})\in \test(\X)$, where $H_{\eps}(\cdot)$ denotes the mollified heat flow $H_{\eps}(u)=\int_0^\infty \eta_{\eps}(s)h_s(u)\d s$ with $\eta\in C^\infty_c(0,\infty)$ (see e.g.\ \cite[Prop.\ 5.2.18]{GP20}). Then  $0< \alpha_{i,n}\le 1$ and $\alpha_{i,n}\to \nchi_{B_i}$ in $L^2(\mm)$ as $n\to \infty$. Up to a subsequence, we can assume that $\alpha_{i,n}\to \nchi_{B_i}$ $\mm$-a.e.\ as well. For all $n\in \nn$ consider a function $g_n\in C^\infty(\rr)$ such that $g_n(t)\ge \max(t,1)$ for all $t\ge 0$ and $g_n(t)\to \max(t,1)$ as $n\to +\infty$ for all $t\in \rr$ (e.g.\ $g_n(t)\coloneqq (1+t+\sqrt{(1-t)^2+1/n})/2$). Define $\tilde \alpha_{i,n}\coloneqq\frac{\alpha_{i,n}}{g_n(\sum_{i=1}^m \alpha_{i,n})}$ which belongs to $\test(\X)$. Note that, since $B_i$ are pairwise disjoint, $g_n(\sum_{i=1}^m \alpha_{i,n})\to 1$ $\mm$-a.e.\ in $B_i$ as $n\to +\infty$, for all $i=1,\dots,m.$ Hence, by the dominated convergence theorem, $\tilde \alpha_{i,n}\to \nchi_{B_i}$ in $L^2(\mm)$ for all $i=1,\dots,m.$ Therefore $v_n\coloneqq \sum_{i=1}^m \tilde \alpha_{i,n} \nabla f_i\in \testv(\X)$ and $v_n\to v$  in $L^2(T\X)$. Finally $|v_n|\le\sum_{i=1}^m \tilde \alpha_{i,n}= \frac{\sum_{i=1}^m \alpha_{i,n}}{g_n(\sum_{i=1}^m \alpha_{i,n})}\le 1$ $\mm$-a.e.

		\noindent {\bf Step 2}: \textit{$v=\sum_{k=1}^N\nchi_{K_k} \nabla f_k$ with $f_k \in \W(\X)$ and pairwise disjoint compact sets $K_k\subset \X$.} 
        Note that $|\d f_k|\le 1$ $\mm$-a.e.\ in $K_k$. We first approximate each term $\nchi_{K_k}\nabla f_k$ separately. Fix $k\in \{1,\dots,N\}$ and set $K\coloneqq K_k$ and $f\coloneqq f_k.$
		Consider $f_n\in \LIP_{bs}(\X)$ such that  $\lip_af_n\to |\d f|$ in $L^2(\mm)$ (and $f_n\to f$ in $\W(\X)$). 
        We have that $\lim_{n\to+\infty}\mm(\{\lip_af_n> 1+\delta\}\cap K)=0.$ Hence, by inner regularity, we can find compact sets $K_n \subset K$ such that $\lip_af_n\le 1+\delta$ in $K_n$ and $\mm(K\setminus K_n)\to 0.$  Since  $K_n$ is compact and by the definition of $\lip_a(f_n)$, for each $n$ we can find finitely many disjoint subsets $B_n^j\subset K_n$, $j=1,\dots,m(n)$, such that $\Lip(f_n\restr{B_n^j})\le 1+2\delta$. We can then extend $f_n\restr{B_n^j}$ to a function $g_n^j\in \LIP_{bs}(\X)$ with $\Lip(g_n^j)\le 1+3\delta$ (see \cite{DiMarino2020GlobalLipschitz}). Therefore the sequence $v_n\coloneqq \sum_{j=1}^{m(n)}\nchi_{B_n^j} \nabla H_{\eps_{j,n}}g_{n}^j $ satisfies $v_n\to \nchi_K \nabla f$  in $L^2(T\X)$ provided we choose each $\eps_{j,n}$ small enough. 
        Additionally $ H_{\eps_{j,n}}g_{n}^j \in \test(\X)$ and  by the Bakry-Émery gradient inequality (see \cite{AmbrosioGigliSavare11-2}) it holds $\Lip( H_{\eps_{j,n}}g_{n}^j )\le (1+4\delta) $ if $\eps_{j,n}$ are chosen small enough. Hence  $(1+4\delta)^{-1}v_n$ satisfies the assumption of Step 1. Repeating this construction for all $K_k$ we find a sequence $(1+4\delta)^{-1}v_n^k$, satisfying the assumption of Step 1, and such that $v_n^k\to \nchi_{K_k} \nabla f_k$  in $L^2(T\X)$. Moreover, all the Borel sets appearing in the definition of all the  $\{v_n^k\}_{n,k}$ are pairwise disjoint. Hence $(1+4\delta)^{-1}w_n\coloneqq (1+4\delta)^{-1}\sum_{k=1}^N v_n^k$ also satisfies the assumption of Step 1  and $w_n\to v$ in $L^2(T\X)$.
       Thus, for all $n$, $(1+4\delta)^{-1}w_n$ can be approximated by a sequence $(w_h^n)_h\subset\testv(\X)$ with $|w_h^n|\le 1+\delta$. By a diagonal argument $v$ is in $L^2(T\X)$-closure of $\{(1+\delta)w_h^n\}_{h,n}$. This concludes the proof (up to decreasing $\delta$).

		\noindent {\bf Step 3}: \textit{$v\in L^2(T\X)$ arbitrarily chosen.} Since $L^2(T\X)$ is generated by $\{\nabla f \ : \ f \in \W(\X)\}$, $v$ can be approximated in $L^2(T\X)$ by elements of the form $v_n=\sum_{k=1}^{N_n} \nchi_{B_k^n}\nabla f_k^n$ with $f_k^n \in \W(\X)$ and $B_k^n\subset \X$ pairwise disjoint Borel sets. Fix $\delta>0.$ For all $n$ we have $\sum_{k=1}^{N_n}\mm(\{|\nabla f_k^n|\ge 1+\delta\}\cap B_k^n)\le \delta^{-2}\|v-v_n\|_{L^2(T\X)}^2.$ Hence, up to  restricting the sets $B_k^n$, we can  assume that $|\nabla f_k^n|\le 1+\delta$ in $B_k^n$ for all $n.$ Finally by inner regularity and up to further restricting the sets $B_k^n$, we can assume that  $B_k^n$ is compact for all $k$ and $n$. Hence  $(1+\delta)^{-1}v_n$ satisfies the assumption of Step 2 for all $n$ and thus can be approximated by a sequence $w_k^n\in\testv(\X)$ with $|w_k^n|\le 1+\delta$. By a diagonal argument $v$ is in $L^2(T\X)$-closure of $\{(1+\delta)w_k^n\}_{k,n}$. This concludes the proof (up to decreasing $\delta$).
	\end{proof}

	Throughout the paper we will need the following version of the coarea formula (cf.\  \cite{Miranda03}).
	\begin{theorem}[Coarea formula]\label{thm:coarea}
		Let $(\X,\sfd,\mea)$ be an $\RCD(K,\infty)$ metric measure space and let $u \in \W_\loc(\X)$. Then $\{u<t\}$ is of locally finite perimeter for a.e.\ $t \in \rr$ and for every Borel function $f: \X\to [-\infty,\infty]$ such that $f \in L^1(|\d u|\mea)$ and all bounded Borel functions $\phi:\rr\to \rr $ it holds
		\begin{equation}\label{eq:coarea main}
			\int \phi(u)f |\d u| \,\d\mea =\int_\rr \phi(t) \int f \d \Per(\{u<t\}) \d t.
		\end{equation}
	\end{theorem}
	\begin{proof}
		Since $\W_\loc(\X)\subset BV_\loc(\X)$ we have that $\{u<t\}$ is of locally finite perimeter for a.e.\ $t \in \rr$.
		We prove the statement for $\phi\ge0 $ and $f\ge0 $, the general case follows taking the positive and negative parts. We can also assume that $f$ has bounded support, otherwise we take $f_n\coloneqq f \nchi_{B_n(x_0)}$ and then pass to the limit by  monotone convergence. Finally we can assume that $u \in \BV(\X)$ by replacing $u$ by $\eta u$ with $\eta \in \LIP_{bs}(\X)$ with $\eta=1$ in the support of $f.$   The coarea formula in \cite[Theorem 2.5]{wu2025almost} then asserts that 
		\begin{equation}
			\int_{\{u>s\}} f |\d u|\d \mm=\int_s^\infty \int f\, \d \pert\, \d t
		\end{equation}
		Hence, since $f|\d u|\in L^1(\mm)$ by assumption,
		\begin{equation}
			\int_{\{s_1<u<s_2\}} f |\d u|\d \mm=\int_{s_1}^{s_2} \int f \d \pert\, \d t, \quad \text{for all $s_1<s_2$}.
		\end{equation}
		In other words \eqref{eq:coarea main} holds with $\phi=\nchi_{[s_1,s_2]}.$ In particular it holds when $\phi$  is a linear combination of such functions.  The statement for a general bounded $\phi$ follows by density using the dominated convergence theorem.
	\end{proof}

	\section{Willmore functional and mean curvature}\label{sec:tandiv and wil}
	Throughout this section $\Xdm$ is an $\RCD(K,\infty)$ space.
	\subsection{Definition of tangential divergence}
	For all $u \in \W(\X)$  we define the co-vector field $e_u \in L^0(T^*\X)$ as
	\[
	e_u\coloneqq \nchi_{\{|\d u|>0\}} \frac{\d u}{|\d u|}.
	\]
	Clearly $|e_u|\le 1$ $\mm$-a.e.

	\begin{definition}[Tangential divergence]\label{def:tandiv}
For all $u\in \W(\X)$ and $v \in \W_C\cap L^\infty (T\X)\cap \dom(\div)$  we define the \textit{tangential divergence} function $\TD(v)\in L^1(|\d u|\mm)$ as
		\[
		\TD(v)\coloneqq \div (v)-\nabla v \left(e_u,e_u\right).
		\]
	\end{definition}
    In particular $\TD(v)$ is defined for all $v\in \testv(\X).$
	The dependence on $u$ is omitted as it always will be clear from context. Note that $\TD(v) \in  L^1(|\d u|\mm)$ because $\{\div(v),\,|\nabla v|,\,|\d u|\}\subset  L^2(\mm)$ and $|e_u|\le 1.$

	In the smooth setting  $\TD(v)$ coincides pointwise on every regular level  set $\{u=t\}$ with the usual tangential divergence of $v$ with respect to $\{u=t\}.$ 
	
	Heuristically, the tangential divergence should depend only on the restriction of the vector field $v$ on the level set $\{u=t\}$. The following simple but crucial result makes this precise in our setting. 
	\begin{prop}\label{prop:cutoff trick}
		For all $v \in \W_C\cap L^\infty (T\X)\cap \dom(\div)$, $u \in \W(\X)$ and all $\phi\in \LIP\cap L^\infty(\rr)$ it holds that $\phi(u)v \in \W_C\cap L^\infty (T\X)\cap \dom(\div)$ and 
		\begin{equation}\label{eq:cutoff trick}
			\TD(\phi(u)v)=\phi(u)\TD(v).
		\end{equation}
	\end{prop}
	\begin{proof}
		The fact that $\phi(u)v \in \W_C\cap L^\infty (T\X)\cap \dom(\div)$ follows by the Leibniz rule for the divergence and the covariant derivative, which also provide the following identities:
		\begin{align*}
			&\div(\phi(u)v)=\phi'(u)\d u(v)+\phi(u)\div(v)\\
			&\nabla (\phi(u)v)=\phi'(u)\nabla u \otimes v + \phi(u) \nabla v.
		\end{align*}
		Therefore
		\begin{align*}
			\TD(\phi(u)v)&=\div(\phi(u)v)-\nabla (\phi(u)v)(e_u,e_u)\\
			&=\phi'(u)\d u(v)+\phi(u)\div(v) - \phi'(u)(\nabla u \otimes v)(e_u,e_u) - \phi(u) \nabla v(e_u,e_u).
		\end{align*}
		The conclusion follows observing that the terms with $\phi'$ simplify because
		\[
		( \nabla u \otimes v)(e_u,e_u)=  \la\nabla u,\e_u\ra \la v,e_u\ra =|\d u|  \la v,e_u\ra=
		\la v,\nabla u\ra=\d u(v).
		\]
	\end{proof}

	\subsection{The Willmore functional}\label{sec:willmore functional}
	
	Our goal is to give a meaning to the quantity
	\begin{equation}\label{eq:formal willmore functional}
		\wil(u)=\frac12\int_\rr \int_{\{u=t\}}|H_t|^2 \d \sigma\d t,
	\end{equation}
	where $H_t$ is the mean curvature of the level set $\{u=t\}$. However, we have yet to introduce the concept of mean curvature, in fact this is one of our main objectives. Our strategy instead is to argue via duality with the tangential divergence, which we defined in the previous section. To explain our approach, recall that in the smooth setting for all sufficiently regular level sets 	and all sufficiently smooth vector fields $v$ it holds
	\begin{equation}\label{eq:formal int by parts}
		\int_{\{u=t\}} \la H_t ,v\ra  \d \sigma =\int_{\{u=t\}} \div_T(v)\d \sigma,
	\end{equation}
	where $ \div_T(v)$ is the tangential divergence in the classical sense. 
This yields the following (formal) identity
\begin{equation}
	\frac12\int_{\{u=t\}}|H_t|^2 \d \sigma=\sup_v 	\int_{\{u=t\}} \la H_t ,v\ra-\frac12|v|^2  \d \sigma=\sup_v \int_{\{u=t\}} \div_T(v)-\frac12|v|^2\d \sigma.
	\end{equation}
Integrating the above identity motivates the following definition.
	\begin{definition}[Willmore functional]\label{def:willmore functional}
		For all $E\subset \rr$ Borel we define the map $\wil_E: \W(\X)\to [0,+\infty] $ as
		\begin{equation}\label{eq:willmorefunctional}
			\wil_E(u)\coloneqq \sup_{v \in \testv(\X)}	\int_{u^{-1}(E)} \left(\TD(v)-\frac{|v|^2}{2}\right) |\d u|  \d\mea\ge0.
		\end{equation}
		We also set $\wil(u)\coloneqq \wil_\rr(u).$
	\end{definition}
	The integral in \eqref{eq:willmorefunctional} makes sense  and it is finite since $\TD(v)\in L^1(|\d u|\mm)$ and $|v|\in L^\infty\cap L^2(\mm).$ 

    Observe that this definition depends only on the tangential divergence and therefore makes sense even if we have not yet introduced any notion of mean curvature.
	We also denote by $\dom(\wil_E)$ the domain of finiteness of $\wil_E$, that is 
	\begin{equation}\label{def:domwil}
		\dom(\wil_E)\coloneqq\{u \in \W(\X) \ : \wil_E(u)<+\infty \}
	\end{equation}
	and for brevity we set $\Dwil\coloneqq \dom(\wil_\rr).$
	At this stage it is not clear whether $\Dwil$ contains non-trivial functions, however we will  prove later  that $\Dwil$ is in fact dense in $L^1(\mm)$ (see Theorem \ref{thm:domain dense}).	Additionally we will prove that a version of identities \eqref{eq:formal willmore functional} and \eqref{eq:formal int by parts} hold also in our setting.

	\subsection{Mean curvature}
	In this part we will show that for functions satisfying $\wil(u)<\infty$, there is a meaningful notion of mean curvature vector $\Hu(u)$ for a.e.\ level set $\{u=t\}.$
	By meaningful we mean that it satisfies the expected integration by parts formula with respect to the tangential divergence and represents the Willmore functional via integration.
    
	 As before, throughout this section  $\Xdm$ is an $\RCD(K,\infty)$ space.  First, we need to introduce a suitable space where the mean curvature vector will live.
	
	$$L^2(T\X;|\d u|\mm)\coloneqq \left\{v \in L^0(T\X) \ : \  |v|\in L^2(|\d u|\mm) \right\}/ \sim,$$
	where $\sim$ denotes the equivalence $|\d u|\mm$-a.e. 
	We endow $L^2(T\X;|\d u|\mm)$ with the norm $\|v\|_{L^2(T\X;|\d u|\mm)}\coloneqq \||v|\|_{L^2(|\d u|\mm)}$.
	Taking the quotient $|\d u|\mm$-a.e.\ is the natural choice because we formally expect to define the mean curvature only on regular level sets, that is where $|\d u|\neq 0$. In particular this set should be interpreted as vector fields defined only along the regular part of the level sets.

	\begin{prop}\label{prop:hilbert}
		
		$\big(L^2(T\X;|\d u|\mm),\|\cdot \|_{L^2(T\X;|\d u|\mm)}\big)$   is a separable Hilbert space and  (the equivalence classes of) vector fields in  $\testv(\X)$ form a dense set. 
	\end{prop}
	\begin{proof}
		From the definitions it follows that $L^2(T\X;|\d u|\mm)$ is a vector space  and $\|\cdot\|_{L^2(T\X;|\d u|\mm)}$ is a norm. Moreover $\|\cdot\|_{L^2(T\X;|\d u|\mm)}$   satisfies the parallelogram identity because the pointwise norm $|\cdot|$ in $L^0(T\X)$ does. We show completeness. Suppose that $v_n$ is a Cauchy sequence  for  $\|\cdot\|_{L^2(T\X;|\d u|\mm)}$ (where $v_n$ denotes also a choice of representative in $L^0(T\X)$). Then for all $k\in \mathbb Z$ the sequence $v_n^k\coloneqq \nchi_{\{|\d u|>1/2^k\}} v_n$ is Cauchy in $L^0(T\X)$ and thus converges to some $v^k\in L^0(T\X)$ as $n\to +\infty.$ We define $v\coloneqq \sum_{k \in \mathbb Z} \nchi_{\{1/2^{k-1}\ge |\d u|>1/2^k\}}v^k \in L^0(T\X)$. Up to passing to a subsequence, we have $|v-v_n|\to 0$ $\mm$-a.e. Then by Fatou's lemma 
		$$\int |v_n-v|^2 |\d u| \d\mm\le \liminf_{m\to +\infty} \int  |v_n-v_m|^2 |\d u| \d\mm\le \eps,$$ 
		for  all $n$ big enough. This shows  both  $|v|\in L^2(|\d u|\mm)$ and $v_n \to v$ in  $L^2(T\X;|\d u|\mm)$.  We pass to the density part. Fix $v \in L^2(T\X;|\d u|\mm)$. Since $v\nchi_{\{|v|\le k\}}\nchi_{\{|\d u|> 1/k\}}\to v$ in $L^2(T\X;|\d u|\mm)$ as $k\to +\infty$ we can assume that  $|v|\in L^\infty\cap L^2(\mm)$ and in particular that $|v|\le 1.$ By Lemma \ref{lem:approximation with bounded testv} there exists a sequence $v_n\in \testv(\X)$ such that  $v_n\to   v$ in $L^2(T\X)$ and $|v_n|\le 2.$ Since $|v_n-v|\to 0$ in $L^2(\mm)$ and $|v_n-v|\le 3$, by the dominated convergence theorem we deduce
		\begin{align*}
	 \int |v_n -v|^2 |\d u|\d \mm\le \||v_n-v|\|_{L^\infty(\mm)} \||v_n-v|\|_{L^2(\mm)} \||\d u|\|_{L^2(\mm)}\rightarrow 0, 
		\end{align*}
		which is what we wanted. The separability follows from the separability of $L^0(T\X)$ (see e.g.\ \cite[(1.3.4)]{Gigli14}).
	\end{proof}

	Below we restate our main result Theorem \ref{thm:main H intro}   on the mean curvature. The existence will be given by an application of the Riesz representation theorem, while some work will be required to obtain \eqref{eq:mean curv eq on level sets} on level sets and the explicit expression \eqref{eq:H formula}.
	\begin{prop}\label{prop:main H} 
		For all $u \in \Dwil$ there exists a unique $\Hu(u) \in L^2(T\X;|\d u|\mm)$, called mean curvature vector, such that for all $v\in \testv(\X)$
		\begin{equation}\label{eq:mean curv eq on level sets}
			\int \la \Hu(u),v\ra \,\d \pert =\int \TD(v)\,\d \pert, \quad \text{for a.e.\ $t \in \rr$.}
		\end{equation}
        Moreover for all $E\subset\rr$ Borel it holds
        \begin{equation}\label{eq:wil integral formula}
            	\wil_E(u)=\frac12\int_{u^{-1}(E)} |\Hu(u)|^2 |\d u|\d \mm=\frac12\int_E \int |\Hu(u)|^2\d \pert\, \d t.
        \end{equation}
        Finally, if $u\in \dom(\Delta)\cap \LIP(\X)$, then $u\in \dom(\wil)$ if and only if 
 \begin{equation}\label{eq:H formula}
     H(u)\coloneqq \left(\frac{\Delta u}{|\d u|} -\frac{\la \nabla |\d u|,\nabla u \ra }{|\d u|^2}\right)\in L^2(|\d u|\mm),
 \end{equation}
in which case $\Hu (u)=H(u)e_u.$
	\end{prop}
Let us stress that in the second part of the statement the function $H(u)$ represents the signed scalar part of the mean curvature, while  $\Hu(u)$ is the vector-valued mean curvature, indeed recall that $|e_u|=1$ $|\d u|\mm$-a.e..

We do not  know however whether $\Hu(u)$ is parallel to $e_u$, that is $\Hu(u)=H(u) e_u$ holds for some function $H(u)$, without the assumption that $u\in \dom(\Delta).$

For the proof of Proposition \ref{prop:main H} we will need  some preliminary results.

	We start introducing an auxiliary  \textit{integral of the tangential divergence}  functional. For all $u\in \W(\X)$  we define the linear map $\Int: \testv(\X)\to \rr$  by
	\begin{equation}
		\Int(v)\coloneqq \int \TD(v) |\d u|\,\d\mm,
	\end{equation}
	which is well defined because $\TD(v)\in L^1(|\d u|\mm)$ by definition.
	
	\begin{lemma}[Boundedness of $\Int$]\label{lem:int tan div bounded} 
		For all $u \in \W(\X)$ it holds
		\begin{equation}
			|  \Int(v)|\le \|v\|_{L^2(T\X;|\d u|\mm)} \sqrt{2\wil(u)}, \quad \forall\,  v \in \testv(\X).
		\end{equation}
	\end{lemma}
	\begin{proof}
		If $\wil(u)=+\infty$ there is nothing to prove, hence we assume that $\wil(u)<+\infty.$
		By the definition of $\wil(u)$, for all $t>0$ it holds
		\[
		\frac1t\wil(u)\ge \frac1t\int \left(\TD(tv)-\frac{|tv|^2}{2}\right) |\d u|  \d\mea=\int \left(\TD(v)-t\frac{|v|^2}{2}\right) |\d u|  \d\mea.
		\]
		Repeating the same for $-v$, we obtain that for all $t>0$
		\[
		|\Int(v)|\le \frac{1}{t}\wil(u)+t\int\frac{|v|^2}{2} |\d u|  \d\mea,
		\]
		from which the result follows by optimizing in $t.$
	\end{proof}

	\begin{prop}\label{prop:H first existence}
		If $u \in \Dwil$ then $\Int$ extends to a linear and continuous functional $\Int \in (L^2(T\X;|\d u|\mm))^*$ satisfying
		\begin{equation}\label{eq:int div norm <= wil}
			\|\Int\|\le \sqrt{2\wil(u)}.
		\end{equation}
		In particular there exists a (unique)   $\Hu(u)\in L^2(T\X;|\d u|\mm)$, called \textit{mean curvature vector}, such that 
		\begin{equation}\label{eq:mean curv def}
			\int \la \Hu(u),v\ra \,|\d u|\d \mm =\Int(v), \quad \forall v \in L^2(T\X;|\d u|\mm).
		\end{equation}
	\end{prop}
	\begin{proof}
		Immediate consequence of Lemma \ref{lem:int tan div bounded} and the Riesz representation theorem recalling that $(L^2(T\X;|\d u|\mm),\|\cdot \|_{L^2(T\X;|\d u|\mm)})$ is a Hilbert space and $\testv(\X)$ is a dense subset by Proposition \ref{prop:hilbert}.
	\end{proof}

For the second part of Proposition \ref{prop:main H} we will need the following technical result. For later use we will state it also in local form.

\begin{lemma}\label{lem:int by parts} 
Let $\Omega\subset \X$ be open.
		Let $v \in \W_C\cap L^\infty (T\X)\cap \dom(\div)$ and let $u \in \LIP_\loc(\Omega)$ be such that $\Delta u \in L^2_\loc(\Omega)$ and  $u^{-1}(0,1)\subset \subset \Omega$.
        Then 
		\begin{equation}\label{eq:diffusetangetialdiv}
			\int_{\{0<u<1\} } \TD(v)|\d u|\d \mea=	\int_{\{0<u<1\}} \la v,e_u\ra \left( \frac{\Delta u}{|\d u|} -\frac{\la \nabla |\d u|,\nabla u \ra }{|\d u|^2}\right)\,|\d u| \d \mea,
		\end{equation}
		where the right integrand is taken to be 0 whenever $|\d u|=0.$ Moreover the same holds if $u\in \dom(\Delta)\cap \LIP(\X)$, without the assumption $u^{-1}(0,1)\subset \subset \Omega$.
	\end{lemma}	
	\begin{proof}
    Note first that both the integrals in \eqref{eq:diffusetangetialdiv} make sense because $\Delta u, |\nabla |\d u||, |\d u|\in L^2_\loc(\Omega)$ and $u^{-1}(0,1)\subset \subset \Omega$ and $|v|, \TD(v)\in L^2(\mm).$ Thanks to $\Delta u \in L^2_\loc(\Omega)$ we have that $\la v,\nabla u\ra\in \W_\loc(\Omega)$ and
	\[
	\nabla v (\nabla u,\nabla u)=\la \nabla \la v,\nabla u\ra, \nabla u\ra -|\d u|  \la \nabla |\d u| ,v\ra, \quad \mea\text{-a.e.\ in $\Omega$}.
	\]
	Fix $\phi \in C^\infty_c(0,1)$. In particular $\phi(u)\in \dom(\Delta)\cap \LIP_c(\X).$ Using the above identity, integrating by parts and rearranging the terms we obtain
	\begin{align*}
		&\int_\Omega \phi(u)|\d u| \div (v)-\frac{\phi(u)}{|\d u|+\eps}\nabla v (\nabla u,\nabla u)\, \d \mea=\\
		&\int_\Omega \left(\frac{|\d u|}{|\d u|+\eps}-1\right) \left( \la \nabla u, v \ra \phi'(u)|\d u|+\la v,\nabla |\d u|\phi(u) \ra \right)+\la v,\nabla u \ra 	\left(\frac{\Delta u}{|\d u|+\eps} -\frac{\la \nabla |\d u|,\nabla u \ra }{(|\d u|+\eps)^2}\right)\phi(u)\, \d \mea.
	\end{align*} 
	We have the following $\mea$-a.e.\ bounds:
	\begin{align*}
			&\left |\frac{\nabla v (\nabla u,\nabla u) }{|\d u| +\eps}\right|\le |\nabla v|_{HS}|\d u| ,
			&& \left |\left(\frac{|\d u| }{|\d u| +\eps}-1\right) \la \nabla u, v \ra |\d u|  \right |
			\le\eps |\d u| |v|, \\
			&\left |\left(\frac{|\d u| }{|\d u| +\eps}-1\right) \la v,\nabla |\d u| \ra  \right |
			\le|\nabla |\d u| ||v|,   &&
			\left |\la v,\nabla u \ra \frac{\Delta u}{|\d u| +\eps} \right|\le |v||\Delta u|,\\ 
			& \left |\la v,\nabla u \ra\frac{\la \nabla |\d u| ,\nabla u \ra }{(|\d u| +\eps)^2}\right|\le |\nabla |\d u| |\,|v|,&&
		\end{align*}
	where all the functions on the right-hand sides are in $L^1_\loc(\Omega)$. Since $\supp(\phi(u))\subset\subset \Omega$ we can apply the dominated convergence theorem sending $\eps\to 0^+$ to obtain
	\[
	\int \phi(u)|\d u| \left(\div (v)-\nabla v \left(\frac{\nabla u}{|\d u|},\frac{\nabla u}{|\d u|}\right)\right)\, \d \mea=
	\int \phi(u) \la v,\nabla u \ra \left( \frac{\Delta u}{|\d u|} -\frac{\la \nabla |\d u|,\nabla u \ra }{|\d u|^2}\right)\, \d \mea,
	\]
	where both integrands are taken to be 0 whenever $|\d u|=0.$
	Finally, taking  a sequence $\phi_n\in C^\infty_c(\rr)$, uniformly bounded in $L^\infty$ and such that $\phi_n\to \nchi_{(0,1)}$ pointwise, plugging $\phi_n$ in the above identity, sending $n\to +\infty$ and applying dominated convergence we reach \eqref{eq:diffusetangetialdiv}.

    The second part of the statement is proved exactly in the same way, observing that all integration by parts still make sense.
	\end{proof}

	We can now proceed with the proof of the existence of the mean curvature vector.
	\begin{proof}[Proof of Proposition \ref{prop:main H}]
		Fix $u\in \Dwil$. We take $\Hu(u)$ as the one given by Proposition \ref{prop:H first existence}. By identity \eqref{eq:mean curv def} and the definition of $\mathcal I_{TV}$, it holds
        \begin{equation}\label{eq:pre int by parts}
            \int \la \Hu(u),v\ra \,|\d u|\d \mm =\int \TD(v) \,|\d u|\d \mm, \quad \forall v \in \testv(\X).
        \end{equation}
        Fix any $\phi\in \LIP_c\cap L^\infty(\rr)$ with $\phi(0)=0.$
        We would like to take $\phi(u)v$ as test vector in the above identity  and apply Proposition \ref{prop:cutoff trick}, however $\phi(u)v$ does not necessarily belong to $\testv(\X)$. Instead, we take a sequence  $f_n\in \test(\X)$ uniformly bounded and such that $f_n\to \phi(u)$ in $\W(\X)$. We claim that $\int \TD(f_nv)|\d u|\d \mm\to \int \TD(\phi(u)v)|\d u|\d \mm$. Indeed by the Leibniz rule
        \[
        \TD(f_nv)=f_n\TD(v)+\la \nabla f_n, v-\la v,e_u\ra e_u\ra
        \]
        and, since $\la \nabla \phi(u), v-\la v,e_u\ra e_u\ra=0$ $\mm$-a.e.,
        \begin{align*}
            \int |\la \nabla f_n, v-\la v,e_u\ra e_u\ra|\, |\d u|\,\d \mm&=  \int |\la \nabla (f_n- \phi(u)), v-\la v,e_u\ra e_u\ra|\, |\d u|\,\d \mm\\
            &\le 2 \||v|\|_{L^\infty(\mm)}\||\d (f_n- \phi(u))|\|_{L^2(\mm)}\||\d u|\|_{L^2(\mm)}\rightarrow 0.
        \end{align*}
        Moreover $\int f_n\TD(v)|\d u|\d \mm \to \int \phi(u)\TD(v)|\d u| \d \mm$ by the dominated convergence theorem. This shows the claim because $\TD(\phi(u)v)=\phi(u)\TD(v)$ by Proposition \ref{prop:cutoff trick}. Taking $f_n v\in \testv(\X)$ in \eqref{eq:pre int by parts} and passing to the limit we obtain 
        \begin{equation}\label{eq:triccata}
              \int \phi(u)\la \Hu(u),v\ra \,|\d u|\d \mm =\int \phi(u)\TD(v) \,|\d u|\d \mm, \quad \forall v \in \testv(\X).
        \end{equation}
		By the coarea formula in Theorem \ref{thm:coarea} we can rewrite \eqref{eq:triccata} as
		\begin{equation}
			\int_\rr \phi(t)\int \la \Hu(u),v\ra \d \pert\,\d t=\int_\rr \phi(t)\int \TD(v)\d \pert\,\d t, \quad \forall v\in \testv(\X).
		\end{equation}
		From \eqref{eq:pre int by parts} we obtain that the above holds in fact for all $\phi\in \LIP_c\cap L^\infty(\rr)$ (not necessarily with $\phi(0)=0$). This gives \eqref{eq:mean curv eq on level sets}. Conversely given any  $\Hu(u) \in L^2(T\X;|\d u|\mm)$ satisfying \eqref{eq:mean curv eq on level sets}, by integrating over $\rr$, using again the coarea formula and the density of $\testv(\X)$, we see that $\Hu(u)$ must satisfy \eqref{eq:mean curv def} as well and thus it is unique by  Proposition \ref{prop:H first existence}.  

For the proof of \eqref{eq:wil integral formula} take any $E\subset \rr$ Borel. Then
\begin{align*}
			\wil_E(u)&=\sup_{v \in \testv(\X)} \int_{u^{-1}(E)} \left(\TD(v)-\frac{|v|^2}{2}\right)|\d u|\d\mm\\
			\text{(by \eqref{eq:mean curv eq on level sets})} \quad &=\sup_{v \in \testv(\X)} \int_{u^{-1}(E)} \left(\la \Hu(u),v\ra-\frac{|v|^2}{2}\right)|\d u|\d\mm= \frac12\int_{u^{-1}(E)} |\Hu(u)|^2|\d u|\d\mm,
		\end{align*}
		where in the last step we used the density of $\testv(\X)$ in $L^2(T\X;|\d u|\mm).$

        We pass to the  second part of the statement. Take any $u\in \dom(\Delta)\cap \LIP(\X)$. Suppose first that $u\in \dom(\wil).$  Then by Lemma
        \ref{lem:int by parts} applied with $\Omega=\X$ and the uniqueness of $\Hu(u)$ it follows that $\Hu(u)=H(u)e_u$ and therefore also $H(u)\in L^2(|\d u|\mm).$ Conversely assume that $H(u)\in L^2(|\d u|\mm).$ Then, again by Lemma \ref{lem:int by parts}  it follows that for all $v\in \testv(\X)$
        \[
        \int \left(\TD(v)-\frac{|v|^2}{2}\right) |\d u| \d \mm \le \int \left(|H(u)||v|-\frac{|v|^2}{2}\right) |\d u| \d \mm\le \frac12\int |H(u)|^2|\d u|\d \mm<\infty.
        \]
        Hence $\wil(u)<\infty$ and the proof is concluded.
	\end{proof}

    \begin{remark}
        In fact identity \eqref{eq:mean curv eq on level sets} holds for any $v\in \W_C\cap L^\infty (T\X)\cap \dom(\div)$ for which there exists a sequence $v_n\in \testv(\X)$ such that $v_n\to v$ in $\W_C(T\X)$ and $\div(v_n)\to \div(v)$ in $L^2(\mm).$  Indeed in this case we would have $\mathcal I_{TV}(v)=\int \TD(v) |\d u|\mm$, which implies that  \eqref{eq:pre int by parts}  holds for  $v$, and the exact same argument would yield \eqref{eq:mean curv eq on level sets}.
        \fr
    \end{remark}

	\subsection{Localizing the construction to open sets}
	Here we adapt the definitions and results of the previous sections to functions not necessarily globally defined. This will be useful, for example, to deal with the electrostatic potential later in the note.
	
	Let $\Omega\subset \X$ be open and let $u\in \W(\Omega)$. We define  
	$e_u\in L^0(T^*\X)\restr{\Omega} $ and $\TD(v)\in L^1(\Omega;|\d u|\mm)$ exactly as we did in Section \ref{sec:tandiv and wil} but restricting to $\Omega.$ Similarly we define the Willmore energy:
	\[
	\wil_E(u)\coloneqq \sup_{v \in \testv(\X)}	\int_{u^{-1}(E)} \left(\TD(v)-\frac{|v|^2}{2}\right) |\d u|  \d\mea\ge0.
	\]

To avoid dealing with $\partial \Omega$ we assume  that  $u\in C(\Omega)$ and
	\begin{equation}\label{eq:compact level sets}
		u^{-1}\big((0,1)\big)\subset \subset \Omega.
	\end{equation}
    
	We can also define $L^2(T\X;|\d u|\mm\restr{\Omega})\coloneqq L^2(T\X)\restr{\Omega} / \sim$, where  $\sim$ is the $|\d u|\mm$-a.e.\ equivalence in $\Omega.$

	The interval $(0,1)$ is taken for simplicity, but can  obviously be replaced by any interval $(a,b).$ 
	
	The following is the local counterpart of Proposition \ref{prop:main H}.
	\begin{prop}\label{prop:main H local} 
		Let $u \in \W\cap C(\Omega)$ be such that \eqref{eq:compact level sets} holds. Suppose that $\wil_{(0,1)}(u)<\infty$. Then  there exists $\Hu(u) \in L^2(T\X;|\d u|\mm\restr{\Omega})$, called mean curvature vector, such that for all $v\in \testv(\X)$
		\begin{equation}\label{eq:mean curv eq on level sets local}
			\int \la \Hu(u),v\ra \,\d {\rm Per}(\{u<t\},\Omega) =\int \TD(v)\,\d {\rm Per}(\{u<t\},\Omega), \quad \text{for a.e.\ $t \in (0,1)$.}
		\end{equation}
		Moreover $\Hu(u)$ is unique $|\d u|\mm$-a.e.\ in $u^{-1}(0,1).$
		
		If in addition $u\in \dom(\Delta,\Omega)\cap \LIP_\loc(\Omega)$, it holds $\wil_{(0,1)}(u)<\infty$ if and only if 
		\begin{equation}\label{eq:H formula local}
		H(u)\coloneqq\frac{\Delta u}{|\d u|} -\frac{\la \nabla |\d u|,\nabla u \ra }{|\d u|^2} \in L^2(u^{-1}(0,1);|\d u|\mm),
		\end{equation}
        in which case $\Hu(u)=H(u)e_u$ $\mm$-a.e.\ in $u^{-1}(0,1).$
	\end{prop}
	\begin{proof}
    We will in fact find $\Hu(u)\in L^2(T\X;|\d u|\mm\restr{A})\coloneqq L^2(T\X)\restr{A} / \sim$, where $A\coloneqq u^{-1}((0,1))$ and $\sim$ is the $|\d u|\mm$-a.e.\ equivalence in $A.$ The required $\Hu(u)\in L^2(T\X;|\d u|\mm\restr{\Omega})$ then follows  by extension to zero.  Exactly as in Proposition \ref{prop:hilbert} we can show that $L^2(T\X;|\d u|\mm\restr{A})$ is a separable Hilbert space and (the restriction to $A$) of $\testv(\X)$ is dense.

    The proof of the statement is then almost identical to the one of Proposition \ref{prop:main H}. We define  the functional 
    \[
    \testv(\X)\mapsto \mathcal I_{TV}^A(v)\coloneqq \int_A \TD(v)|\d u|\, \d \mm.
    \]
    The boundedness argument of Lemma \ref{lem:int tan div bounded}  applies verbatim thanks to the assumption $\wil_{(0,1)}(u)<\infty$, yielding by the Riesz representation theorem a unique $\Hu(u)\in L^2(T\X;|\d u|\mm\restr{A})$ which represents $\mathcal I_{TV}^A(v)$. To obtain identity \eqref{eq:mean curv eq on level sets local} on level sets we note that for all $\phi\in \LIP_c(0,1)$ the local analogue of Proposition \ref{prop:cutoff trick} still applies. Then, the same density and coarea argument of Proposition \ref{prop:main H} gives \eqref{eq:mean curv eq on level sets local}. 

    Finally, when $u\in \dom(\Delta,\Omega)\cap \LIP_\loc(\Omega)$ the last assertion follows exactly as in Proposition \ref{prop:main H}, using the local version of Lemma \ref{lem:int by parts}.
	\end{proof}
	
 It turns out that harmonic functions have finite Willmore energy.
	\begin{cor}\label{cor:H formula for harmonic functions}
    Let $\Xdm$ be an $\RCD(K,N)$ space with $N\in(2,\infty).$
		Let $u$ be harmonic in $\Omega$  and such that $u^{-1}(0,1)\subset\subset \Omega$. Then $\wil_{(0,1)}(u)<\infty$ and 
		\[
		\Hu(u) =-\frac{\la \nabla |\d u|,\nabla u \ra }{|\d u|^2} e_u, \quad \mm\text{-a.e.\ in $u^{-1}(0,1).$}
		\]
	\end{cor}
	\begin{proof}
	    By Theorem \ref{cor:harmest2} we have that $\nchi_{\{|\d u|>0\}}\frac{|\d |\d u||^2}{|\d u|} \in L^1_\loc(\Omega),$ which implies precisely that \eqref{eq:H formula local} is satisfied.
	\end{proof}
	In the next section, building upon the above result and cut-off arguments, we will construct many functions that have finite Willmore energy in the global sense.

	\subsection{Existence of many functions with finite Willmore energy}
    Our goal is to prove the following density result.
	\begin{theorem}\label{thm:domain dense}
		Let $\Xdm$ be an $\RCD(K,N)$ space with $N\in[2,\infty)$. Then $\Dwil\cap \LIP_c(\X)\cap D(\Delta)$ is dense in $L^p(\mm)$ for all $p\in [1,\infty).$
	\end{theorem}
	For the proof we need the following existence result  for harmonic functions. The proof  can be done using \cite[Theorem B.1]{GigliViolo23} and a barrier argument, for  details see \cite[Corollary 2.2.4]{thesis}.
	\begin{theorem}[Existence of harmonic cut-off functions]\label{cor:harmonic cut off 2}
		Let $(\X,\sfd,\mea)$ be an $\mathsf{RCD}(K,N)$  space with $N<+\infty$. For every compact set $P\subset \X$ and any  open set $U\subset \X$ containing $P$ there exists a continuous function $u \in C(\X)\cap \W_0(U)$ such that $0\le u\le 1$ and
		\[
		\begin{cases}
			\Delta u=0,  & \text{in $\{0<u<1\}$},\\
			u=1, & \text{in $P$},\\
			u=0, & \text{in $\X\setminus U$}.
		\end{cases}
		\]
	\end{theorem}
	
	For the proof of Theorem \ref{thm:domain dense} we need a further technical lemma.
	\begin{lemma}\label{lem:phi}
		Let $u$ be harmonic in $\Omega$. Then for every $\phi \in C^2(\rr)$ with $\phi',\phi''$ bounded it holds that $\phi(u)\in \LIP_{\loc}(\Omega)\cap \dom(\Delta,\Omega)$, $|\d \phi(u)|\in \W_{\loc}(\Omega)$ and
		\[
		\int_{\Omega'} \left|\frac{\Delta \phi(u)}{|\d \phi(u)|}-\frac{\la \nabla |\d \phi(u)| ,\nabla\phi(u) \ra}{|\d \phi(u)|^2}\right|^2 |\d \phi(u)| \d \mea<+\infty,
		\]
		for every $\Omega'\subset \subset \Omega$, where the whole integrand is taken to be 0 whenever  $ |\d \phi(u)| =0$.
	\end{lemma}
	\begin{proof}
		We first prove the regularity claims on $\phi(u)$.
		Recall that being harmonic $u$ is locally Lipschitz (see \cite{Jiang13}), hence $\phi(u)\in \LIP_\loc(\Omega)$ and from the chain rule for the Laplacian we have that $\phi(u)\in \D(\Delta,\Omega)$ with $\Delta (\phi(u))=\phi''(u)|\d u|^2\in L^2_\loc(\Omega).$ From the chain rule for the gradient we have $|\d \phi(u)|=|\phi'(u)||\d u|$ $\mea$-a.e.\ in $\Omega.$   Moreover, since $\phi'(u)\in \LIP_{\loc}(\Omega)$, again from the chain rule for the gradient we have $|\phi'(u)|\in \LIP_\loc(\Omega)$ with $\d |\phi'(u)|=\sgn(\phi'(u))\phi''(u)\d u$. Recall also that from Theorem \ref{cor:harmest2} we have that $|\d u|\in \W_{\loc}(\Omega)$. Combining these observations with the Leibniz rule for the gradient  we obtain that $|\d \phi(u)|\in \W_{\loc}(\Omega)$ with 
		\begin{align*}
			\d |\d \phi(u)|=\sgn(\phi'(u))\,\big(\phi''(u)\d u|\d u|+\phi'(u)\d|\d u|\big), \quad \mea\text{-a.e.\ in $\Omega$.}
		\end{align*}
		Therefore
		\begin{align*}
			\frac{\Delta \phi(u)}{|\d \phi(u)|}-\frac{\la \d |\d \phi(u)| ,\d \phi(u) \ra}{|\d \phi(u)|^2}=-\sgn(\phi'(u))\frac{\la \d |\d u| ,\d u \ra}{|\d u|^2}, \quad \mea\text{-a.e.\ in $\Omega$}
		\end{align*}
		where both sides are taken to be $0$ whenever $|\d u|=0$ or $|\d \phi(u)|=0.$ Combining this with \eqref{eq:harmest2} gives the desired conclusion.
	\end{proof}
	We now proceed with the main density result.
	\begin{proof}[Proof of Theorem \ref{thm:domain dense}]
		For any set $K\subset \X$ and number $\eps>0$ we define the open tubular neighborhood $K^{\eps}\coloneqq\{x \ : \ \sfd(x,K)<\eps\}$.
		
		\textbf{Step 1}: We show that for every compact set $K\subset \X$ and every $\eps>0$ there exists $u_{K}^\eps \in\Dwil\cap \LIP_c(\X)\cap D(\Delta)$ such that $u_{K}^\eps=1$ in $K$, $0\le u_{K}^\eps\le 1$ and $\supp(u_{K}^\eps)\subset K^\eps$.

		Fix any such $K \subset \X$ and $\eps>0$.  From Theorem \ref{cor:harmonic cut off 2} we know that there exists a function $u=u_\eps \in\ C(\X)\cap \W_0(K^\eps)$ with $0\le u\le 1$, $u=1$ in $K$, $\supp(u)\subset K^\eps$ and which is harmonic in $\Omega\coloneqq\{0<u<1\}$. Choose a function $\phi \in C^2(\rr)$ so that $0\le \phi\le 1$, $\phi\equiv 0$ in $(-\infty,1/3)$ and $\phi \equiv 1$ in $(2/3,+\infty)$. It holds $0\le \phi(u)\le 1$, $\phi(u)=1$ in $K$ and $\supp (\phi(u))\subset K^\eps.$ It remains to show that $\phi(u) \in \Dwil\cap \LIP_c(\X)\cap D(\Delta)$. From the properties of $u$ it follows that $u^{-1}([1/3,2/3])\subset \subset \Omega$, since it is a closed set and $\partial \Omega\subset \{u=1\}\cup \{u=0\}$ (because $u$ is continuous in $\X$). In particular, since $u \in \LIP_{\loc}(\Omega)$, we have $\phi(u)\in \LIP_c(\X)$. Moreover from Lemma \ref{lem:phi} we have that $\phi(u)\in\dom(\Delta,\Omega)$ and $|\d \phi(u)|\in \W_{\loc}(\Omega)$. Additionally from the locality property of the gradient and Laplacian we have that 
		$$|\nabla |\d \phi(u)||=0=\Delta\restr{\Omega}\phi(u), \quad \mea\text{-a.e.\ in $\Omega\setminus \{1/3\le u\le 2/3\}$}.
		$$
		In particular, again because $u^{-1}([1/3,2/3])\subset \subset \Omega,$ from Lemma \ref{lem:phi} we obtain that
		\[
		\int_{\Omega} \left|\frac{\Delta\restr{\Omega} \phi(u)}{|\d \phi(u)|}-\frac{\la \nabla |\d \phi(u)| ,\nabla\phi(u) \ra}{|\d \phi(u)|^2}\right|^2 |\d \phi(u)| \d \mea<+\infty.
		\]
		Thanks to  the second part of Proposition \ref{prop:main H} it suffices to show that $\phi(u) \in \dom(\Delta)$. Indeed this would imply also that $\phi(u)\in \Dwil.$  To show this let $\eta \in \LIP_\c(\X)$ be arbitrary and let $\eta'\in \LIP_c(\Omega)$ be such that $\eta' \equiv 1$ in $\supp \eta \cap \{1/3\le u\le 2/3\}$. From the linearity of the gradient operator
		\begin{align*}
			\int \la\nabla \phi (u), \nabla \eta\ra\, \d \mea &= \int \la\nabla  \phi (u), \nabla (\eta\eta')\ra\, \d \mea +\int_{ \X\setminus \{1/3\le u\le 2/3\}} \la\nabla  \phi (u), \nabla (\eta(1-\eta'))\ra\, \d \mea \\
			&= \int \la\nabla \phi(u), \nabla (\eta\eta')\ra\, \d \mea,
		\end{align*}
		where in the second equality we used that $|\d \phi(u)|=0$ $\mea$-a.e.\ in $\X\setminus \{1/3\le u\le 2/3\}$.  Moreover since now $\eta\eta' \in \LIP_c(\Omega)$ we can use that $\phi(u) \in \dom(\Delta,\Omega)$  to deduce that
		\[
		\int \la\nabla \phi(u), \nabla \eta\ra\, \d \mea =- \int \eta' \eta \Delta \phi(u) \d \mea=- \int_{\Omega} \eta' \eta \Delta\restr{\Omega} \phi(u) \, \d \mea, 
		\]
		which implies that $\phi(u) \in \dom(\Delta,\X)$ with $\Delta \phi(u)=\nchi_\Omega \Delta\restr{\Omega} \phi(u)$. However, since $\Delta\restr{\Omega}\phi(u)=0$ in $\Omega\setminus \{1/3\le u\le 2/3\}$ we also deduce that $\Delta \phi(u) \in L^2(\X)$, which is what we wanted.
		
		\textbf{Step 2}: We prove that given $\{K_i\}_{i=1}^n$  disjoint compact subsets of $\X$, then $\sum_{i=1}^n \alpha_i\nchi_{K_i}$ is in the $L^p(\mm)$-closure of  $\Dwil\cap \LIP_c(\X)\cap D(\Delta)$ for every $\alpha_i \in \R$. Indeed take a sequence $\eps_k \downarrow 0$ so that the open sets $K_i^{\eps_k}$ are still disjoint for all $k$ and $i=1,\dots n,$ and let $u_{K_i}^{\eps_k}\in \Dwil\cap \LIP_c(\X)\cap D(\Delta)$ be the functions given in Step 1. Then
		\[
		f_{k}\coloneqq \sum_{i=1} \alpha_i u_{K_i}^{\eps_k} \in \Dwil\cap \LIP_c(\X)\cap D(\Delta).
		\]
		Note that $\Dwil$ is not in general closed by linear combinations, however in this case the supports of $u_{K_i}^{\eps_k}$ are pairwise disjoint, hence it is immediately checked that $f_k \in  \Dwil.$ Clearly $u_{K_i}^{\eps_k}\to \nchi_{K_i}$ as $\eps_k\to 0$ in $L^p(\mm)$, hence 
		$$f_k \to \sum_{i=1}^n \alpha_i\nchi_{K_i}, \text{ in $L^p(\mm)$}.$$
		
		\textbf{Step 3}:  Given $\{E_i\}_{i=1}^n$  disjoint Borel subsets of $\X$ with $\mm(E_i)<\infty$, then $\sum_{i=1}^n \alpha_i\nchi_{E_i}$ is in the $L^p(\mm)$-closure of  $\Dwil\cap \LIP_c(\X)\cap D(\Delta)$ for every $\alpha_i \in \R$. By inner regularity we have that for every $i=1,\dots,n$ there exists a sequence of compact sets $K_i^k \subset E_i$ , $k \in \nn$, such that $\nchi_{K_i^k}\to \nchi_{E_i}$  in $L^p(\mm)$ as $k \to +\infty$. Hence
		\[
		\sum_{i=1}^n \alpha_i\nchi_{K_i^k}\to \sum_{i=1}^n \alpha_i\nchi_{E_i}, \text{ in $L^p(\mm)$}.
		\]
		Hence we conclude by Step 2.
		
		\textbf{Step 4}: Every $f\in L^p(\mm)$ is in  the  $L^p(\mm)$-closure of  $\Dwil\cap \LIP_c(\X)\cap D(\Delta)$. This is immediate from Step 3, since the functions $\sum_{i=1}^n \alpha_i\nchi_{E_i}$ with  $\{E_i\}_{i=1}^n$  disjoint Borel subsets of $\X$ and $\alpha_i \in \R$ are dense in $L^p(\mm).$
	\end{proof}

	\section{Lower semicontinuity  of the Willmore functional}\label{sec:lsc}

Here we show Theorem \ref{thm:lsc} saying that the Willmore functional $\wil$ is lower- semicontinuous under pmGH-convergence of ambient spaces and strong $\W$-convergence of the functions. This result was announced in \cite[Section 6.6]{Gigli23_working}.

In the rest of this section we assume to have fixed a sequence of pmGH-converging $\RCD(K,\infty)$ (pointed) spaces $(X_n,\sfd_n,\mm_n,x_n)$, $n\in \mathbb N\cup \{\infty\}$ via extrinsic approach, that is by assuming each $(\X_n,\sfd_n)$ isometrically embedded in a common metric space $(Y,\sfd)$ such that $x_n\to x_\infty$ and $\mm_n\rightharpoonup\mm_\infty$ in duality with $C_{bs}(Y)$  (see \cite{GigliMondinoSavare13,AmbrosioHonda17} for additional details). We also assume the reader to be familiar with the notion of $L^2$-convergence and $\W$-convergence of functions under pmGH convergence (see \cite{AmbrosioHonda17,Gigli23_working}).

We will denote by $\wil_n$ the Willmore functional in the space $\X_n$ and rewrite Theorem \ref{thm:lsc} below for the convenience of the reader.
	
	\begin{theorem}[$\Gamma-\liminf$ of the Willmore functional]\label{thm:lsc 2}
		Suppose that $u_n\to u_\infty$ in $\W$. Then
		\[
		\wil_\infty(u_\infty)\le \liminf_n \wil_n(u_n).
		\]
	\end{theorem}

For the proof we need the notion of test sequences introduced in \cite{Gigli23_working} (see also \cite[Section 19]{AmbrosioHonda17}).

	\begin{definition}[Test sequences of functions and tensors]\label{def:test sequence}
		We say that $n \mapsto f_n \in \mathrm{Test}(X_n)$, $n \in \mathbb{N} \cup \{\infty\}$, is a \emph{test sequence} if $\sup_n \|f_n\|_{L^\infty(\mm_n)} + \mathrm{Lip}(f_n) + \|\Delta f_n\|_{L^\infty(\mm_n)} < \infty$ and
		\begin{equation}
			\begin{aligned}
				f_n &\xrightarrow{L^2} f_\infty \\
				\Delta f_n &\xrightarrow{L^2} \Delta f_\infty
			\end{aligned}
			\qquad \text{and} \qquad
			\begin{array}{l}
				\mathrm{Ch}_n(f_n) \to \mathrm{Ch}_\infty(f_\infty) \\
				\sup_n \mathrm{Ch}_n(\Delta f_n) < \infty.
			\end{array}
		\end{equation}
		Then, $n \mapsto z_n \in L^2(TX_n)$ is a \emph{test sequence of vectors} if it is of the form $z_n = \sum_{i=1}^m f_{i,n} \nabla g_{i,n}$ for $(f_{i,n})$, $(g_{i,n})$ test sequences.
		
		More generally, for $d > 1$ and $n \mapsto z_n \in L^2(T^{\otimes d} X_n)$, $n \in \mathbb{N} \cup \{\infty\}$ we say that $(z_n)$ is a \emph{test sequence of tensors} if it is of the form $z_n = \sum_{i=1}^m f_{i,n} \nabla g_{1,i,n} \otimes \cdots \otimes \nabla g_{d,i,n}$ for $(g_{j,i,n})$ and $(f_{i,n})$ test sequences.
	\end{definition}
	
	\begin{definition}[$L^2$-convergence of tensors]
		We say that $n \mapsto v_n \in L^2(T^{\otimes d} X_n)$ \emph{converges weakly in $L^2$} to $v_\infty \in L^2(T^{\otimes d} X_\infty)$,  $v_n\rightharpoonup v_\infty$ in $L^2$ for short, provided
		\[
		\sup_n \|v_n\|_{L^2(T^{\otimes d} X_n)} < \infty
		\qquad \text{and} \qquad
		\lim_n \int \langle v_n, z_n \rangle \, \mathrm{d} m_n = \int \langle v_\infty, z_\infty \rangle \, \mathrm{d} m_\infty
		\]
		for any test sequence $(z_n)$ of $d$-tensors.
		
		The convergence is \emph{strong},  $v_n\to  v_\infty$ in $L^2$ for short, if moreover $\|v_n\|_{L^2(T^{\otimes d} X_n)} \to \|v_\infty\|_{L^2(T^{\otimes d} X_\infty)}$.
	\end{definition}
    It is observed in \cite[d) p.\ 78]{Gigli23_working} that $z_n\to z_\infty$ in $L^2$ for any test sequence of tensors $z_n.$
	
\begin{lemma}\label{lem:approx by test vectors}
For any $v\in \testv(\X_\infty)$ there exist test sequences of vectors $(z_n^k)_n$ such that $z_\infty^k\to v$ in $\W_C(T\X)$ and $\div(z_n^k)\to \div(v)$ in $L^2(\mm_\infty)$ as $k\to +\infty.$
\end{lemma}
\begin{proof}
    Let $v=\sum_{i=1}^m f_i\nabla g_i \in \testv(\X_\infty)$. Since $f_i,g_i\in \LIP(\X_\infty)$, we can extend both to bounded functions in $\LIP(Y).$ Then $z_n^k\coloneqq \sum_{i=1}^m\tilde  h_{n,1/k}(f_i)\nabla \tilde h_{n,1/k}(g_i)$ are test sequences of vectors as shown in \cite[f) pag.\ 78]{Gigli23_working} ($\tilde h_{n,1/k}$ denotes the mollified heat flow in the space $\X_n$). The fact that $z_\infty^k\to v$  in $\W_C(T\X)$  and $\div(z_n^k)\to \div(v)$ in $L^2(\mm_\infty)$ can be checked using that $\Delta \tilde h_{\infty,1/k} f_i\to \Delta  f_i$, $\Delta \tilde h_{\infty,1/k} g_i\to \Delta  g_i$  in $L^2(\mm_\infty)$ and that  $\tilde h_{\infty,1/k} f_i\to f_i$,  $\tilde h_{\infty,1/k} g_i\to g_i$ in $\W(\X_\infty)$.
\end{proof}

	We collect in the following result several properties of convergence of tensors that we will need in the sequel.
	\begin{prop}[Convergence properties of tensors]\label{prop:conv prop of tensors}
		Let $v_n,w_n \in L^2(T^{\otimes d}\X_n)$, $f_n\in L^2(\mm_n)$, $n\in \mathbb N\cup \{\infty\}.$ Then the following hold:
		\begin{enumerate}[label=\roman*)]
			\item if $v_n \to v_\infty$ in $L^2$ and  $f_n\to f_\infty$ in $L^2$, $|f_n|\le C<\infty$, then $f_nv_n\to f_\infty v_\infty$ in $L^2;$
            \item  if $v_n \to v_\infty$ in $L^2$ then $|v_n|\to |v_\infty|$ in $L^2$;
            \item  if $\sup_n \||v_n|\|_{L^2(\mm_n)}<\infty$ then, up to a subsequence, $v_n\rightharpoonup v$ in $L^2$;
			\item if $v_n \to v_\infty$ in $L^2$,  $w_n\to w_\infty$ in $L^2$ and $|w_n|\le C$, then $v_n\otimes w_n\to v_\infty\otimes w_\infty$ in $L^2$;
			\item if $v_n\rightharpoonup v_\infty$ in $L^2$ and $w_n\to w_\infty$ in $L^2$ then $\la v_n,w_n\ra  \mm_n\rightharpoonup \la v_\infty, w_\infty\ra \mm_\infty $ as measures; if also  $|v_n|\le C$, then  $\la v_n,w_n\ra  \rightharpoonup \la v_\infty, w_\infty\ra  $ in $L^2$ and if additionally $v_n\to v_\infty$ and $|w_n|\le C$ then $\la v_n,w_n\ra \to \la v_\infty,w_\infty\ra $ in $L^2$;
			\item if $f_n\to f_\infty$ in $\W$ then $\nabla f_n\to \nabla f_\infty$ in $L^2.$
		\end{enumerate}
	\end{prop}
    \begin{proof}
        ii), iii), vi), and the  first part of v)  are proved in \cite[Section 6.1]{Gigli23_working}. For the second part of v) note that if $|v_n|\le C$ then $\sup_n \|\la v_n,w_n\ra \|_{L^2(\mm_n)}<\infty$ which combined with the first part of v) implies weak $L^2$-convergence by definition. If also $v_n\to v_\infty$ and $|w_n|\le C$ we have 
        $$\la v_n,w_n\ra=\frac14 (|v_n+w_n|^2-|v_n-w_n|^2)$$
        and the right-hand side converges strongly in $L^2$. Indeed $|v_n\pm w_n|$ converge in $L^2$ by linearity and ii) and are bounded in $L^\infty$, hence  $|v_n\pm w_n|^2$ converge in $L^2$ (see \cite[Prop.\ 2.18-viii)]{NobiliViolo21}).  
        To prove i) fix any test sequence of vectors $z_n$. By item v) $\la v_n,z_n\ra$ converges weakly in $L^2$ to $\la v_\infty,z_\infty\ra$.  Hence $\int f_n\la v_n,z_n\ra\d \mm_n$ converges to $\int f_\infty\la v_\infty,z_\infty\ra\d \mm_\infty$. By iii) and thanks to $|f_n|\le C$, up to passing to a subsequence, $f_nv_n$ converges weakly in $L^2$ to some $w,$ which by the previous observation must coincide with $f_\infty v_\infty.$  Moreover by ii), $|v_n|\to |v_\infty|$ strongly in $L^2$. Hence $f_n|v_n|\to f_\infty|v_\infty|$ strongly in $L^2$. Indeed the product of two $L^2$-strongly converging sequences of functions, one uniformly bounded in $L^\infty$, converges strongly in $L^2$ as well (it follows from the properties of $L^2$-convergence with fixed measure which carry to this setting as explained  in \cite[page 28]{Gigli23_working}). Therefore $\|f_n|v_n|\|_{L^2(\mm_n)}\to \|f_\infty |v_\infty|\|_{L^2(\mm_\infty})$, which yields the required strong $L^2$-convergence.
         
        To prove iv) fix a  {test sequence of tensors}  $z_n = \sum_{i=1}^m f_{i,n} \nabla g_{1,i,n} \otimes \cdots \otimes \nabla g_{2d,i,n}$. Then
        \[
         \la v_n\otimes w_n , z_n\ra =  \sum_{i=1}^m   \la v_n,  f_{i,n} \nabla g_{1,i,n} \otimes \cdots \otimes \nabla g_{d,i,n} \ra \la w_n ,  \nabla g_{d+1,i,n} \otimes \cdots \otimes \nabla g_{2d,i,n}\ra
        \]
        for all $n\in \mathbb N\cup \{\infty\}.$
      The sequences $  \la v_n,  f_{i,n} \nabla g_{1,i,n} \otimes \cdots \otimes \nabla g_{d,i,n} \ra$ and $ \la w_n ,  \nabla g_{d+1,i,n} \otimes \cdots \otimes \nabla g_{2d,i,n}\ra$ converge respectively weakly in $L^2$ and strongly in $L^2$ to the corresponding expressions for $n=\infty$, by item v) for all $i$. Hence
      \[
      \int  \la v_n\otimes w_n , z_n\ra\d \mm_n\to  \int  \la v_\infty\otimes w_\infty , z_\infty\ra \d \mm_\infty.
      \]
      This shows that $v_n\otimes w_n \rightharpoonup v_\infty\otimes w_\infty $ in $L^2$. On the other hand $|v_n\otimes w_n|=|v_n||w_n|$ which converge strongly in $L^2$ to $|w_\infty ||v_\infty|=|v_\infty\otimes w_\infty|$ because of ii) and $|w_n|\le C$. This proves $v_n\otimes w_n \to v_\infty\otimes w_\infty $ in $L^2$.
    \end{proof}

	\begin{lemma}\label{lem:prop of test sequence}
		Let $v_n\in L^2(T \X_n)$, $n\in \nn\cup \{\infty\}$ be a test sequence of vectors. Then, up to passing to a subsequence,  $|v_n|\to |v_\infty|$ in $L^2$, $\div(v_n)\to \div(v_\infty)$ in $L^2$ and $\nabla v_n\rightharpoonup \nabla v_\infty$ in $L^2$.
	\end{lemma}
	\begin{proof}
    The fact that $|v_n|\to |v_\infty|$ in $L^2$ follows from Proposition \ref{prop:conv prop of tensors}-ii). Instead $\div(v_n)\to \div(v_\infty)$ in $L^2$ is proved in \cite[e) pag.\ 78]{Gigli23_working}. It remains to show that $\nabla v_n\rightharpoonup \nabla v_\infty$ in $L^2$ up to passing to a subsequence.
		We have $v_n=\sum_{i=1}^m f_i^n\nabla g_i^n$ with $f_i^n,g_i^n $  as in Definition \ref{def:test sequence}. Hence by the Leibniz rule for the covariant derivative
		$\nabla v_n= \sum_{i=1}^m \nabla f_i^n \otimes \nabla  g_i^n +f_i^n {\rm Hess}( g_i^n)$.
		Moreover by the assumptions on the sequence $g_i^n$
		\[
		\sup_n \int |{\rm Hess}( g_i^n)|^2 \d \mm_n \le \int |\Delta g_i^n|^2 - K|\nabla g_i^n|^2\d \mm_n<\infty.
		\]
		This shows that $\sup_n \||\nabla v_n|\|_{L^2(\mm_n)}<\infty$ and the result follows from \cite[(ix) in p.\ 79]{Gigli23_working}.
	\end{proof}

	\begin{lemma}\label{lem:booh}
		Let $v_n\in L^2(T^{\otimes 2}\X_n)$ be such that $v_n\rightharpoonup v_\infty$ in $L^2$ and suppose $u_n\to u_\infty$ in $\W.$ Then
		\begin{equation}
			\int \la v_n, \nabla u_n\otimes \tfrac{\nabla u_n}{|\d u_n|}\ra \d \mm_n \to  \int \la v_\infty, \nabla u_\infty \otimes \tfrac{\nabla u_\infty}{|\d u_\infty|}\ra \d \mm_\infty,
		\end{equation}
		where both integrands are taken to be zero respectively when $|\d u_n|=0$ and $|\d u_\infty|=0$.
	\end{lemma}
	\begin{proof}
		Let us first show that for all $\delta>0$ and $\phi \in \LIP_{bs}(Y)$ it holds
		\begin{equation}\label{eq:delta psi conv tensors}
			\int \phi \la v_n, \nabla u_n\otimes \tfrac{\nabla u_n}{|\d u_n|+\delta}\ra \d \mm_n \to  \int \phi \la v_\infty, \nabla u_\infty \otimes \tfrac{\nabla u_\infty}{|\d u_\infty|+\delta}\ra \d \mm_\infty.
		\end{equation}
		
		It is sufficient to show that $\frac{\phi}{|\d u_n|+\delta}\to \frac{\phi}{|\d u_\infty|+\delta}$ in $L^2$. Indeed  combining  i)-iv)-vi) in Proposition \ref{prop:conv prop of tensors} would imply that $\phi \nabla u_n\otimes \frac{\nabla u_n}{|\d u_n|+\delta} \to  \phi \nabla u_\infty \otimes \frac{\nabla u_\infty}{|\d u_\infty|+\delta}$ and then  \eqref{eq:delta psi conv tensors} would follow  by v) in Proposition \ref{prop:conv prop of tensors}.   The $L^2$-convergence follows observing that  $\frac{1}{|\d u_n|+\delta}-\frac{1}{\delta}\to \frac{1}{|\d u_\infty|+\delta}-\frac1\delta$ in $L^2$, indeed $\psi(t)\coloneqq \frac{1}{t+\delta}-\frac{1}{\delta} \in \LIP(\rr)$ and $\psi(0)=0$ (apply e.g.\ \cite[Prop.\ 3.3-a)]{AmbrosioHonda17}). Since $L^2$-convergence is linear and stable under multiplication by functions in $\LIP_{bs}(Y)$ (see \cite[Prop.\ 3.3-c)]{AmbrosioHonda17}) we conclude that \eqref{eq:delta psi conv tensors} holds.
		
		Next observe that for all $n\in \nn\cup \{\infty\}$ we have
		\[
		\left|\nabla u_n\otimes \tfrac{\nabla u_n}{|\d u_n|+\delta} - \nchi_{\{|\d u_n|>0\}}\nabla u_n\otimes \tfrac{\nabla u_n}{|\d u_n|}\right |\le 
		\tfrac{\delta |\d u_n|}{|\d u_n|+\delta}\le \delta, \quad \text{$\mm_n$-a.e.}
		\]     
		Since $\sup_n \||v_n|\|_{L^2(\mm_n)}<\infty$ and $\supp(\phi)$ is bounded,  we conclude that \eqref{eq:delta psi conv tensors} holds also with $\delta=0.$ To remove $\phi$ take any $\eps>0$ and  choose $\phi$ so that $\|(1-\phi)|\d u_\infty|\|_{L^2(\mm_\infty)}<\eps.$ Since $|\d u_n|\to |\d u_\infty|$ in $L^2$ we have $(1-\phi)|\d u_n|\to (1-\phi)|\d u_\infty|$ in $L^2$ as well. Therefore by \eqref{eq:delta psi conv tensors} for $\delta=0$ we get
		\begin{align*}
			& \left  | \int \la v_n, \nabla u_n\otimes \frac{\nabla u_n}{|\d u_n|}\ra \d \mm_n-  \int \la v_\infty, \nabla u_\infty \otimes \frac{\nabla u_\infty}{|\d u_\infty|}\ra \d \mm_\infty\right |\\
			&\le \limsup_n  \int \left|(1-\phi) \la v_n, \nabla u_n\otimes \frac{\nabla u_n}{|\d u_n|}\ra \right| \d \mm_n +  \int \left|(1-\phi) \la v_\infty, \nabla u_\infty\otimes \frac{\nabla u_\infty}{|\d u_\infty|}\ra \right| \d \mm_\infty\\
			&\le \limsup_n \||v_n|\|_{L^2(\mm_n)} \|(1-\phi)|\d u_n|\|_{L^2(\mm_n)}+ \||v_\infty|\|_{L^2(\mm_\infty)} \|(1-\phi)|\d u_\infty|\|_{L^2(\mm_\infty)}\le C\eps,
		\end{align*}
		where $C$ does not depend on $\eps.$
		By the arbitrariness of $\eps$ we conclude.
		
	\end{proof}

	\begin{proof}[Proof of Theorem \ref{thm:lsc 2}]
		Thanks to Lemma \ref{lem:approx by test vectors} it is sufficient to show that for any test sequence of vector fields $v_n$ (as in Definition \ref{def:test sequence}), up to passing to a subsequence, we have 
		\begin{equation}
			\begin{split}
					\int \left(\div(v_n)-\nabla v_n(e_{u_n},e_{u_n})-\tfrac{|v_n|^2}{2}\right)& |\d u_n|  \d\mea_n \\
					&\to 	\int \left(\div(v_\infty)-\nabla v_\infty(e_{u_\infty},e_{u_\infty})-\tfrac{|v_\infty|^2}{2}\right) |\d u_\infty|  \d\mea_\infty,
			\end{split}
		\end{equation}
		where $e_{u_n}=\frac{\nabla u_n}{|\d u_n|}\nchi_{\{|\d u_n|>0\}}.$
		Since $|v_n|\to |v_\infty|$ in $L^2$ and $|v_n|\le C$ we have $|v_n|^2\to |v_\infty|^2$ in $L^2$. Moreover $|\d u_n|\to |\d u_\infty|$ in $L^2$. Hence $\int |v_n|^2|\d u_n|\d \mm_n\to \int |v_\infty|^2|\d u_\infty|\d \mm_\infty.$ By Lemma \ref{lem:prop of test sequence}, up to passing to a subsequence, we have $\div(v_n)\to \div(v_\infty)$ in $L^2$ and similarly we get     $\int \div(v_n)|\d u_n|\d \mm_n\to \int \div(v_\infty)|\d u_\infty|\d \mm_\infty.$ Finally, again by Lemma \ref{lem:prop of test sequence}, $\nabla v_n\rightharpoonup \nabla v_\infty$ in $L^2$. Hence applying Lemma \ref{lem:booh} we get precisely that $\int \nabla v_n(e_{u_n},e_{u_n})|\d u_n|\d \mm_n\to \int \nabla v_\infty(e_{u_\infty},e_{u_\infty}) |\d u_\infty|\d \mm_\infty.$
	\end{proof}

	\section{Regularity of Willmore energy on level sets of harmonic functions}\label{sec:reg along level sets}
	
	The main objective of this part is to prove the following regularity result for the Willmore energy.
	
	\begin{prop}[Regularity of the Willmore energy]\label{prop:wil lsc} Let $\Xdm$ be an $\RCD(K,N)$ space, $N\in(2,\infty).$
		Let $u$ be harmonic in $\Omega$  such that $u^{-1}(0,1)\subset\subset \Omega$. Then for all $\beta \in [\tfrac{N-2}{N-1},1]$ the  function
		\begin{equation}\label{eq:p wil on level sets harm}
			t\mapsto \int |\Hu (u)|^{1+\beta}\,{\rm Per}(\{u<t\},\Omega), \quad \text{a.e.\ $t\in(0,1)$}
		\end{equation}
		admits a lower semicontinuous representative.
	\end{prop}
	Note that the function \eqref{eq:p wil on level sets harm} is well defined up to taking a.e.\ equivalence classes.

Throughout this section $\Xdm$ is an $\RCD(K,N)$ space, $N\in(2,\infty).$
    
	For the proof we need two preliminary lemmas.
	\begin{lemma}[Continuity of weighted gradient]\label{lem:w12}
		Let $u$ be harmonic in $\Omega$ and satisfying \eqref{eq:compact level sets}. For all $g \in \W_\loc\cap L^\infty_\loc(\Omega)$  and all $\alpha \ge 1+\frac{N-2}{N-1}$ the function
		\[
		(0,1)\ni t \mapsto \int |\d u|^\alpha g \,\d {\rm Per}(\{u<t\},\Omega),
		\]
		belongs to $W^{1,1}_\loc(0,1)$ and thus has a continuous representative.
	\end{lemma}
	\begin{proof}
		Fix $\phi \in C^1_c(0,1)$ with $0\le \phi \le 1$. We can extend $u$ to a function in $\dom(\Delta)\cap \LIP(\X)$ agreeing with $u$ in $\supp(\phi(u))$ (see e.g.\ the proof of Theorem \ref{thm:domain dense}). Similarly we extend $g$ to a function in $\W\cap L^\infty(\X)$ with $\supp(g)$ compact in $\Omega$ and agreeing with $g$ in $\supp(\phi(u))$. Then $g|\d u|^{\alpha}\in L^1(|\d u|\mm)$. Hence by the coarea formula in Theorem \ref{thm:coarea} and integrating by parts
		\begin{align*}
			&-\int \phi'(t) \int |\d u|^\alpha g \,\d{\rm Per}(\{u<t\},\Omega)\d t= - \int \phi'(u) |\d u|^{\alpha+1} g \d \mm\\
			&=-\int \phi'(u)\la \nabla u,\nabla u \ra|\d u|^{\alpha-1} g\d \mm= -\int \la \nabla \phi(u),\nabla u \ra |\d u|^{\alpha-1} g \d \mm\\
			&= \int \phi(u) \div( g |\d u|^{\alpha-1}\nabla u)\,\d \mm.
		\end{align*}
		Since $u$ is harmonic, by Theorem \ref{cor:harmest2} we have that $|\d u|^{\alpha-1}\in \W_\loc(\Omega)$ and so $g|\d u|^{\alpha-1}\in \W(\X).$
		By the Leibniz  rule for the divergence we have
		\[
		\div( g |\d u|^{\alpha-1}\nabla u)=g|\d u|^{\alpha-1}\Delta u + \la \nabla (g|\d u|^{\alpha-1}),\nabla u\ra \in L^1(\mm).
		\]
		Hence  $\div( g |\d u|^{\alpha-1}\nabla u)=h |\d u|$ for some $h\in L^1(|\d u|\mm).$ Hence,  again by the coarea formula we have both 
		\[
		-\int \phi'(t) \int |\d u|^\alpha g \,\d {\rm Per}(\{u<t\},\Omega)\d t= \int \phi(t) \int h \d {\rm Per}(\{u<t\},\Omega) \d t
		\]
		and $ \int h \d \pert \d t\in L^1(\rr)$. Since both sides of the  above identity do not depend on the extension of $u$ and $g$, we conclude.
	\end{proof}

	\begin{lemma}[BV regularity of mean curvature]\label{lem:bv}
		Let $u$ be harmonic in $\Omega$ and satisfying \eqref{eq:compact level sets}. Then for all $g \in \W_\loc\cap L^\infty_\loc(\Omega)$ and all $\beta \ge \frac{N-2}{N-1}$ the function
		\[
		f(t)\coloneqq   \int H(u) |\d u|^{\beta} g \d {\rm Per}(\{u<t\},\Omega),
		\]
		belongs to $\BV_\loc(0,1)$ and its jump points are contained in the countable set
		\begin{equation}\label{eq:jump points}
			\cal S_{u,\beta}\coloneqq \{t \in \rr \ : \ |\bd |\d u|^\beta|(\{u=t\})>0\}.
		\end{equation}
	\end{lemma}
	\begin{proof}
		Set $P_t \coloneqq {\rm Per}(\{u<t\},\Omega)$ as measure.
		Fix $\phi \in C^1_c(\rr)$ with $\supp(\phi)\subset K$ with $K\subset (0,1)$ compact. By the assumptions $\phi(u)\in \LIP_c(\Omega).$ Recall that $|\d u|^\beta \in \dom(\bd,\Omega)$ for all $\beta\ge \frac{N-2}{N-1}$ (see Theorem \ref{cor:harmest2}). By the coarea formula in Theorem \ref{thm:coarea} (applied after an extension of $u$ outside $K$ to a function in $\LIP_c(\X)$) and integrating by parts
		\begin{align*}
			&-\beta\int_0^1 \phi'(t) \int H(u) |\d u|^{\beta}g \d P_t \d t= -\beta \int \phi'(u) H(u) |\d u|^{\beta+1} g \d \mm\\
			&= \beta\int \phi'(u)\la \nabla u,\nabla |\d u|\ra |\d u|^{\beta-1} g\d \mm= \int \la \nabla \phi(u),\nabla |\d u|^\beta\ra g \d \mm\\
			&= -\int \phi(u)g \bd |\d u|^\beta\d \mm- \int \phi(u) \la \nabla g, \nabla |\d u|^\beta\ra \d \mm.
		\end{align*}
		Therefore
		\begin{align*}
			\left|-\int_0^1 \beta\phi'(t) \int H(u) |\d u|^{\beta}g \d P_t \d t\right|&\le \|\phi\|_\infty \left(\|g\|_\infty  \,|\bd |\d u|^\beta| (u^{-1}(K)) + \|g\|_{\W(\X)}\||\nabla |\d u|^\beta|\|_{L^2(u^{-1}(K))} \right)\\
			&=C_{\Omega,g,K,u}\,\|\phi\|_\infty.
		\end{align*}
		This shows $ f\in \BV_\loc(0,1)$. Let now $Df$ be its distributional derivative. Let $\bar t \in (0,1)$ be a jump point, i.e.\ such that $|Df(\{\bar t\})|>0$. Let $\phi_n\in C^1_c(0,1)$ be such that $\phi_n(\bar t)=1$, $0\le \phi_n\le 1$ and $\supp (\phi_n)\subset (\bar t-1/n,\bar t+ 1/n).$ Then
		\begin{align*}
			&0<|D f(\{\bar t\})|=\lim_n \left |\int_0^1 \phi'_n f \d t \right |\\
			&\le \beta^{-1}\liminf_n \left |\int \phi_n(u)g \bd |\d u|^\beta\d \mm\right| + \beta^{-1}\int_{\{\bar t-1/n\le u\le \bar t+ 1/n\}} |\la \nabla g, \nabla |\d u|^\beta \ra| \d \mm =\\
			&\le \beta^{-1}\liminf_n \left |\int \phi_n(u)g \bd |\d u|^\beta \d \mm\right|\le \beta^{-1}\|g\|_\infty| \bd (|\d u|^\beta)|(\{u= \bar t\}).
		\end{align*}
		This completes the proof of the second part.
	\end{proof}
	For all $g \in  \W_\loc\cap L^\infty_\loc(\Omega)$ and all $q\ge 2$ define the functions $R_{g^q},R_{Hg}(t): (0,1)\to \rr$ as the continuous (resp.\ minimal upper semicontinuous) a.e.\ representative of
	\[
	R_{g^q}(t)\coloneqq \int |\d u|^q|g|^q\,\d{\rm Per}(\{u<t\},\Omega)\quad  \bigg(\text{ resp.\  $\quad R_{Hg}(t)\coloneqq \int H(u)|\d u|g\,\d{\rm Per}(\{u<t\},\Omega)$}\bigg), 
	\]
	which exist by Lemma \ref{lem:w12} and Lemma \ref{lem:bv}.

	\begin{proof}[Proof of Proposition \ref{prop:wil lsc}]
		Set $p\coloneqq 1+\beta \in (1,2]$  and $q\coloneqq \frac{p}{p-1}\ge 2.$ Let $H(u)\in L^2(\Omega;|\d u|\mm)$ be the function such that $\Hu(u)=H(u)e_u$ as in Proposition \ref{prop:main H local}.
		Consider  $h_n \in \LIP_c(\X)$ such that $h_n \to -|H(u)|^{p-2}H(u)$ both in $L^q(|\d u|\mm)$ and $(|\d u|\mm)$-a.e.\ and such that $|h_n|\le G \in L^q(|\d u|\mm)$ (where $|\d u|$ and $H(u)$ are extended to zero outside $\Omega$).  Set $g_n\coloneqq (1/n+|\d u|)^{-1}h_n\in \W_\loc\cap L^\infty_\loc(\Omega)$.
		Define the lower semicontinuous function $\widetilde W_p(t):(0,1)\to \rr$ as
		\[
		\widetilde W_p\coloneq \sup_{n}   \left(  -pR_{Hg_n}-\frac pq R_{g_n^q}\right).
		\]
		By Young's inequality for all $n$ it holds
		\begin{equation}\label{eq:young's trick}
			-pR_{Hg_n}-\frac pq R_{g_n^q}\le  \int |H(u)|^p\,\d{\rm Per}(\{u<t\},\Omega), \quad \text{for a.e.\ $t\in(0,1).$}
		\end{equation}
		In particular
		\begin{equation}\label{eq:Rhp}
			\widetilde W_p\le  \int |H(u)|^p\,\d{\rm Per}(\{u<t\},\Omega), \quad  \text{for a.e.\ $t\in(0,1).$}
		\end{equation}
		By the dominated convergence theorem we can check that
		\[
		\int_{u^{-1}(a,b)} \left(\frac1p|\Hu(u)|^p+H(u)|\d u|g_n+\frac1q|\d u|^q|g_n|^q \right)|\d u|\mm \to 0,
		\]
		for all $(a,b)\subset \subset  (0,1)$. Note also that the functions $|\d u|^{q}|g_n|^q, |\Hu(u)|^p$ and $H(u)g_n$, all belong to $L^1_\loc(\Omega;|\d u|\mm)$. Hence (after an extension of $u$ to $\X$) we can apply the coarea formula in Theorem \ref{thm:coarea}  to rewrite the above   as
		\[
		\int_a^b  \left| \int |\Hu(u)|^p\,\d\pert+pR_{Hg_n}(t)+\frac pqR_{g_n^q}(t)\right|\,\d t \to 0,
		\]
		where the modulus comes from \eqref{eq:young's trick}.
		Hence, up to passing to a subsequence, the integrand converges pointwise to zero in $(a,b)\setminus N$, for some $N$ null set. Therefore, up to enlarging $N,$ thanks to \eqref{eq:Rhp}
		\begin{align*}
			\widetilde W_p(t)&=  \sup_{n}   \left(  -pR_{Hg_n}(t)-\frac pq R_{g_n^q}(t)\right) =\lim_{n}   \left(  -pR_{Hg_n}(t)-\frac pq R_{g_n^q}(t)\right)\\
			&=\int |\Hu(u)|^p\,\d{\rm Per}(\{u<t\},\Omega), \quad \forall t \in (a,b)\setminus N.
		\end{align*}
		This shows that $\int |\Hu(u)|^p\, \d{\rm Per}(\{u<t\},\Omega)$ coincides a.e.\ in $(a,b)$ with $\widetilde W_p$.
	\end{proof}
	From now on we denote by $W_p$ the (pointwise) maximal lower semicontinuous representative\footnote{
		Recall that for any $f:(0,1)\to \rr$, the l.s.c.\ envelope $f_*(x)\coloneqq \sup_{\delta>0} \essinf_{B_\delta(x)} f$ is l.s.c.,  $f_*\le f$ a.e.\ and it is the  (pointwise) maximal with these properties.  In particular if $f$ admits a l.s.c.\ a.e.\ representative $\tilde f$, then $\tilde f\le f_*$, $f_*=f$ a.e.\ and so $f_*$ is the (pointwise) maximal with this property.
	} of $\int |\Hu(u)|^p\d\pert$ in $(0,1).$
	
	\section{Isocapacitary inequality and Willmore inequality}\label{sec:isocap will}
	
		\subsection{The isocapacitary inequality via rearrangement}
	The goal of this section is to prove the following isocapacitary inequality.
	\begin{theorem}[Sharp isocapacitary inequality]\label{thm:isocap}
		Let $\Xdm$ be a $\cd(0,N)$ space with $N\in (2,\infty)$ and $\avr(\X)>0$. Then for every bounded Borel set $E \subset \X$   it holds
		\begin{equation}\label{eq:isocap}
			\Cap(E)\ge(N-2)N \omega_N^{\frac2N} \avr(\X)^{\frac2N} \mea(E)^\frac{{N-2}}{N}.
		\end{equation}
        Moreover if equality holds for some $E$ with $\mm(E)>0$ and $\Xdm$ is an $\RCD(0,N)$ space, then  $E$ coincides $\mm$-a.e.\ with a ball centred at $x_0$ and  $\X$ is an $N$-Euclidean cone with tip $x_0$.
	\end{theorem}
	We will prove this inequality via symmetrization, using a  {P{\'o}lya}-{Szeg{\H{o}}}-type inequality proved in \cite{NobiliViolo24} (which is a variant of the one proved in \cite{MondinoSemola20}).
	
	To state it  we  first recall the notion of distribution function and rearrangement using the terminology from \cite{NobiliViolo24} (see also \cite{NobiliViolo2025}).

	Let  $\Xdm$ be a metric measure space  and $u:\X\to [0,\infty]$ be a Borel function such that $\mm( \{u>t\})<\infty$ for all $t>0$. We define  $\mu:[0,+\infty)\to[0,\infty)$, the distribution function of $u$ as  $\mu(t)\coloneqq \mm(\{u> t\})$. For $u$ and $\mu$ as above, we consider the generalized inverse $u^{\#}:[0,\infty)\to [0,\esssup u]$ of $\mu$ defined by
	\begin{equation*}
		u^{\#}(s)\coloneqq 
		\begin{cases*}
			{\rm ess}\sup u & \text{if $s=0$},\\
			\inf\left\lbrace t > 0 \ :\ \mu(t)<s \right\rbrace &\text{if $s>0$}.
		\end{cases*}
	\end{equation*}
	Note that $u^{\#}$ is monotone non-increasing, left-continuous.
	
	The monotone non-increasing rearrangement $u^* : [0,\infty) \rightarrow \R^+$ is defined by
	\[ u^*(t) := u^\#(\omega_Nt^N ), \qquad \forall\, t \in [0,\infty).\]
	In particular $u^*$ is monotone non-increasing. The functions $u$ and $u^*$ are equimeasurable:
	\begin{equation}\label{eq:equimeas}
		\mm(\{u>t\})=\mm_N(\{u^*>t\}), \quad \forall\, t>0,
	\end{equation}
    where $\mm_N\coloneqq N\omega_N t^{N-1}\d t.$

	\begin{theorem}[Euclidean {P{\'o}lya}-{Szeg{\H{o}}} inequality and rigidity]\label{thm:polya}
		Let $(\X,\sfd,\mm)$ be a ${\cd}(0,N)$ space with $N \in(1,\infty)$ and $\avr(\X)>0$. Let $u\in \W_\loc(\X)$ be non-negative with $|\d u|\in L^2(\mm)$ and such that  $\mm( \{u>t\})<\infty$ for all $t>0$.
		Then  $u^*\in {\sf AC}_\loc(0,\infty)$ and
		\begin{equation}\label{eq:PZ}
			\int |\d u|^2\, \d \mm \ge {\sf AVR}(\X)^\frac{2}{N} \int |u'|^2\, \d \mm_N(t).
		\end{equation}
        Moreover if $\Xdm$ is an $\RCD(0,N)$ space and equality holds in \eqref{eq:PZ} with both sides non-zero and $(u^*)'\neq 0$ a.e.\ in $\{\essinf u<u^*<\esssup u \}$, then  there exists $x_0$ such that
       \begin{equation}
           u=u^*\circ (\avr(\X)^\frac{1}{N}\sfd_{x_0}), \quad \mm-a.e.
       \end{equation}
       and  $\X$  is an $N$-Euclidean cone with tip $x_0.$
	\end{theorem}

	The last ingredient for the proof of the isocapacitary inequality is the following computation relative to the capacity in the model one-dimensional space.
	\begin{lemma}\label{lem:cap in IN}
		Consider the one-dimensional metric measure space $\Xdm=([0,\infty),\sfd_{eu},\mm_N)$ for $N>2.$ Then 
		\[
		\Cap([0,r))= (N-2)N \omega_N^{\frac2N}\mea_N([0,r))^\frac{{N-2}}{N}, \quad \forall \, r>0.
		\]
	\end{lemma}
	\begin{proof}
		We only  consider $r=1$, as the general case  follows by rescaling. It can be checked using the definitions that the function   $u(x)\coloneqq \min(x^{2-N},1)$ is the electrostatic potential for $[0,1].$ 
		Since $([0,\infty),\sfd_{eu},\mm_N)$ is an $\rcd(0,N)$ space  with Euclidean volume growth, from Proposition \ref{prop:electro cap formula} we have 
		\[
		\Cap([0,1])= N\omega_N\int_1^\infty |u'|^2(t)t^{N-1}\d t=N\omega_N (N-2)^2 \int_1^\infty t^{1-N} \d t=N\omega_N (N-2).
		\]
	\end{proof}

	We can now prove the isocapacitary inequality.
	\begin{proof}[Proof of Theorem \ref{thm:isocap}]
    We can assume that $\mm(E)>0$, otherwise there is nothing to prove.
    For the moment we consider the case that $(\X,\sfd,\mm)$ is an $\RCD(0,N)$ space. Since $E$ is bounded it holds $\Cap(E)<\infty.$  By Proposition \ref{prop:electro cap formula} there exists $u\in \W_\loc(\X)$ such that $0\le u\le 1$, $u\ge 1$  $\Cap$-a.e.\ in $E$, vanishing at infinity and such that
        \[
        \Cap(E)=\int_\X |\d u|^2\d \mm.
        \]
        Since $u$ vanishes at infinity we have $\mm(\{u>t\})<\infty$ for all $t>0$. 
		Hence we can apply the {P{\'o}lya}-{Szeg{\H{o}}} inequality given in Theorem \ref{thm:polya} to obtain that $u^*\in {\sf AC}_\loc(0,\infty)$ and 
		\begin{equation}\label{eq:gradu > u'}
			\int |\d u|^2\, \d \mea\ge \avr(\X)^{\frac2N} \int | (u^*)'|^2\,\d \mea_N.
		\end{equation}
		 By equimeasurability,
		\begin{align*}
			\mea_N(\{u^*\ge 1\})&=\lim_{t\uparrow 1}\mea_N(\{u^*> t\})=\lim_{t\uparrow 1}\mea(\{u> t \})\\
			&=\mea(\{u\ge 1\}) \ge  \mea(E),
		\end{align*}
		where the last inequality follows from $u \ge 1$ in $E$ $\mm$-a.e. Since $u^*$ is continuous and monotone non-increasing, $\{u^*\ge 1\}=[0,\bar r]$ for some $\bar r>0.$  Therefore, by Lemma \ref{lem:cap in IN} for all $r<\bar r$ it holds
		\begin{equation}\label{eq:u*'>cap}
			\int | (u^*)'|^2\,\d \mea_N\ge \Cap([0,r])=(N-2)N \omega_N^{\frac2N}\mea_N([0,r))^\frac{N-2}N.
		\end{equation}
		Combining \eqref{eq:u*'>cap} with \eqref{eq:gradu > u'} and sending $r\uparrow \bar r$ we conclude recalling that $\mm_N([0,\bar r])\ge \mm(E).$ Suppose now that equality holds in \eqref{eq:isocap}. Then equality holds in \eqref{eq:gradu > u'} and also in \eqref{eq:u*'>cap} with $r=\bar r$.   From the latter and Proposition \ref{prop:electro cap formula} we deduce that $u^*$ must be the electrostatic potential of $[0,\bar r]$ and thus $(u^*)'\neq 0$ in $\{0<u^*<1\}$. Hence we can apply the second part of Theorem \ref{thm:polya} and deduce that $u$ is $\mm$-a.e.\ equal to a radial function around a point $x_0.$  Moreover we must have the equality $\mm_N([0,\bar r])= \mm(E)$, which implies $\mea(\{u\ge 1\})=\mm(E).$ Since $E\subset \{u\ge 1\}$ up to $\mm$-null sets we deduce that $E=\{u\ge 1\}$ up to $\mm$-null sets. Since $u$ is  $\mm$-a.e.\ equal to a radial function we conclude that $\{u\ge 1\}$ is a ball  up to $\mm$-null sets which concludes the proof.

        If $(\X,\sfd,\mm)$ is only a $\CD(0,N)$ space we can not use Proposition \ref{prop:electro cap formula}. Instead, we apply the rearrangement procedure to all non-negative functions $u\in \W(\X)$ with bounded support that are $\ge 1$ $\mm$-a.e.\ in a neighborhood of $E$. Passing  to the infimum, recalling ii) in Lemma \ref{lem:cap bs} and using again Lemma \ref{lem:cap in IN} we obtain \eqref{eq:isocap}. This procedure, however, does not lead to a rigidity statement as above.
	\end{proof}

	\subsection{Willmore inequality and rigidity}
	In this part we consider $\Xdm$ an $\RCD(0,N)$ space, $N\in(2,\infty),$ with $\avr(\X)>0$ and  $u$  the electrostatic potential of a compact set $K\subset \X$ (as in Definition \ref{def:electro}). 
	
	For any $\beta\ge \frac{N-2}{N-1}$ we define
	\begin{equation}\label{eq:defU}
		U_{\beta}(t)\coloneqq\frac{1}{t^{\beta \frac{N-1}{N-2}}}\int |\d u|^{\beta+1}\, \d \Per(\{u<t\}) \, \in L^1_{\loc}(0,1),
	\end{equation}
	which by Lemma \ref{lem:w12} admits a locally absolutely continuous representative, still denoted by $U_{\beta}$. Note that under the current assumptions, $\pert={\rm Per}(\{u<t\},\X\setminus K)$ as measures for all $t\in(0,1)$, hence we will use $\pert$ for brevity.

	It turns out that $U_\beta$ is monotone non-decreasing. For Riemannian manifolds this fact was  discovered by Colding \cite{Co12} in the case of the Green function  and later proved in \cite{AgostinianiFogagnoloMazzieri20} in general form. This monotonicity was generalized  to the setting of $\RCD$ spaces in \cite{GigliViolo23}, which we state here in a simplified and weaker form although sufficient for our purposes.
	
	\begin{theorem}[Monotonicity of $U_\beta$]\label{thm:monotonicity}
		Let $\X$ be an ${\rm RCD}(0,N)$  space, $N>2,$ with Euclidean volume growth and let  $u$ be the electrostatic potential of a compact set $K$. Let $U_{\beta}$ be the function defined in \eqref{eq:defU}, with $\beta \ge \frac{N-2}{N-1}$. Then $U_{\beta}$ is monotone non-decreasing in $(0,1)$.
		
		Moreover $U_\beta'\in BV_\loc(0,1)$,
		\begin{equation}\label{eq:Uprimo}
			U_\beta'(t)=\beta t^{-\frac{\beta(N-1)}{N-2}}\left[\int -H(u)|\d u|^{\beta} -\frac{N-1}{N-2} \frac1t  |\d u|^{\beta+1} \d\pert   \right] , \quad a.e.\ t\in (0,1)
		\end{equation}
		and the maximal lower semicontinuous representative of $U_\beta'$ satisfies
		\begin{equation}\label{eq:quant mon}
			t^2U_\beta'(t)\ge C_{\beta,N} \int_{\{u<t\} }u^2 |\d |\d u^\frac{1}{2-N}|^\frac{\beta}{2}|^2\d \mm, \quad \text{for all } t\in(0,1),
		\end{equation}
		where $C_{\beta,N}=\frac{4}{\beta} \left(\beta-\frac{N-2}{N-1}\right)$ and the right-hand side is taken to be zero when $\beta=\frac{N-2}{N-1}.$
	\end{theorem}
	\begin{proof}
		The monotonicity of $U_\beta$ and \eqref{eq:quant mon} are proved in \cite[Theorem 5.4]{GigliViolo23}. Moreover in \cite[Proposition 5.5]{GigliViolo23} it is shown that
		\[
		U_\beta'(t)=\int \la \frac{\nabla u}{|\d u|},\nabla \left(\frac{|\d u|^\beta}{u^{\beta\frac{N-1}{N-2}}}\right)\ra \d \pert, \quad a.e.\ t\in (0,1).
		\]
		Formula \eqref{eq:Uprimo} then follows developing the gradient term, recalling both that $|\d u|^\beta\in \W_\loc(\Omega)$ by Theorem \ref{cor:harmest2} where $\Omega= \X \setminus K$, and that $H(u)=-\frac{\la \nabla |\d u|,\nabla u \ra }{|\d u|^2} $ (as defined in Corollary \ref{cor:H formula for harmonic functions}). Finally \eqref{eq:quant mon} appears in \cite[Theorem 5.4]{GigliViolo23} for the left-continuous representative of $t^2U_\beta'(t)$. However, from \cite[Prop.\ 5.6]{GigliViolo23} we know that the distributional derivative of $t^2U_\beta'(t)$ is non-negative, hence this left-continuous representative coincides with the maximal lower semicontinuous representative. \end{proof}
	
	In the following, $\sigma_N\coloneqq N\omega_N.$
	\begin{prop}[Lower bound for $U_\beta$]\label{prop:lower bound U}
		Let $\X$ be an $\rcd(0,N)$ space, $N>2$, with Euclidean volume growth, $u$ be an electrostatic potential for a compact set $K$ and $U_{\beta}$ be the continuous representative of the  function defined in \eqref{eq:defU}, with $\beta \ge \frac{N-2}{N-1}$.
		
		Then
		\begin{equation}\label{eq:lower bound U}
			U_{\beta}(t)\ge (N-2)^{\beta \frac{N-1}{N-2}}(\avr(\X)\sigma_N)^{\frac{\beta}{N-2}}{\rm Cap}(K)^{\frac{N-2-\beta}{N-2}}, \quad \forall t \in(0,1).
		\end{equation}
	\end{prop}
	\begin{proof}
		We argue by contradiction and we assume that  for some $\theta<\avr(\X)$ and some $\bar t \in(0,1)$
		\begin{equation}\label{eq:contradiction ass}
			U_{\beta}(\bar t)\le (N-2)^{\beta \frac{N-1}{N-2}}(\theta\sigma_N)^{\frac{\beta}{N-2}}\Cap(K)^{\frac{N-2-\beta}{N-2}}
		\end{equation}
		Before starting the proof we  fix once and for all a Borel representative of $|\d u|$. 
		
		We now make the preliminary observation that for all $\eps>0$, by the coarea formula (Theorem \ref{thm:coarea}) applied after an extension of $u$ gives
		\[
		\int_\eps^{1-\eps} \int \frac{1}{|\d u|}\d  \Per(\{u<s\})\d s =\int_{\{\eps<u<1-\eps\}}\nchi_{\{|\d u|>0\}}\d \mea \le \mea(\{u>\eps\})<+\infty,
		\]
		where $1/|\d u|$ is defined to be zero whenever $|\d u|=0$. 
		Hence $\int \frac{1}{|\d u|}\d \Per(\{u<s\})<+\infty$ for a.e. $s \in(0,1).$ Thus  there exists at least one
		$t_0 \in (0,\bar t)$ such that  $\int \frac{1}{|\d u|}\d \Per(\{u<t_0\})<+\infty$ and $U_\beta(t_0)=(t_0)^{-\beta\frac{N-1}{N-2}}\int |\d u|^{\beta+1}\, \d \Per(\{u<t_0\}).$ Moreover 
		\begin{equation}\label{eq:contr assumption 2}
			\begin{split}
				t_0^{-\beta\frac{N-1}{N-2}}\int |\d u|^{\beta+1}&\, \d \Per(\{u<t_0\})=U_\beta(t_0)\le U_\beta(\bar t)\\
				&\overset{\eqref{eq:contradiction ass}}{\le} (N-2)^{\beta \frac{N-1}{N-2}}(\theta\sigma_N)^{\frac{\beta}{N-2}}\Cap(K)^{\frac{N-2-\beta}{N-2}},
			\end{split}
		\end{equation}
		where in the first inequality  we used the monotonicity of $U_\beta$.
		
		For convenience of notation, for every $t \in(0,1)$ we set
		\[
		u_t \coloneqq t^{-1}u, \quad K_t\coloneqq \X \setminus\{u< t\}.
		\]
		Observe that  $u_t$ is the electrostatic potential for $K_t$. Since the map $(0,1)\ni t\mapsto \int |\d u| \d \pert$ is a.e.-equal to a constant (see \cite[Proposition 5.7]{GigliViolo23}),  by Proposition \ref{prop:electro cap formula} we have
		\begin{align*}
			\Cap(K_t)&=\int_{\{u<t\}} |\d u_t|^2\d \mea=\frac1{t^2}\int_0^t \int |\d u|\d \Per(\{u<s\})\d s\\
			&=\frac1t\int_{\{u<1\}} |\d u|^2\d \mea=\frac1t \Cap(K), \quad \text{ for all $t \in (0,1)$.}
		\end{align*}
		Therefore
		\begin{equation}\label{eq:Urewritten}
			\begin{split}
				U_\beta(t)&=\frac{1}{t^{\beta\frac{N-1}{N-2}}} \int |\d u|^{\beta+1}\, \d \pert= \frac{1}{t^{\beta\frac{N-1}{N-2}-(\beta+1)}} \int |\d u_t|^{\beta+1}\, \d \pert=\\ &=\left(\frac{\Cap(K_t)}{\Cap(K)}\right)^{-\alpha} \int |\d u_t|^{\beta+1}\, \d \pert, \quad \text{ for a.e.\ $t \in (0,1)$,}
			\end{split}
		\end{equation}
		where $\alpha\coloneqq -(\beta\frac{N-1}{N-2}-(\beta+1))=\frac{N-2-\beta}{N-2}.$
		
		Moreover from  H\"older's inequality with exponents $p=\frac{\beta+2}{2},$ $p'=\frac{\beta+2}{\beta}$  we get
		\begin{align*}
			\Cap(K_t)&=\int |\d u_t|\d \pert=\int |\d u_t|^{1+\frac{\beta}{\beta+2}}|\d u_t|^{-\frac{\beta}{\beta+2}}\d \pert\\
			&\le \left(\int |\d u_t|^{\beta+1}\d \pert \right)^{2/(\beta+2)} \left(\int \frac{1}{|\d u_t|}\d \pert \right)^{\beta /(\beta+2)},
		\end{align*}
		for a.e.\ $t\in(0,1)$ such that also $\int \frac{1}{|\d u_t|}\d \pert <+\infty$ (recall that almost every $t$ has this property). The above can be rewritten as
		\begin{equation}\label{eq:holder magic}
			\left(\int |\d u_t|^{\beta+1}\d \pert \right)^{2/\beta}\ge \Cap(K_t)^{\beta+2/\beta}\left(\int \frac{1}{|\d u_t|}\d \pert \right)^{-1}.
		\end{equation}
		Independently, from the monotonicity of $U_\beta$  coupled with the expression of $U_\beta$ given in \eqref{eq:Urewritten} we have
		\[
		(\Cap(K_{t_0}))^{-\alpha}\int |\d u_{t_0}|^{\beta+1} \d \Per(\{u<t_0\})\ge (\Cap(K_t))^{-\alpha}\int |\d u_{t}|^{\beta+1} \d \Per(\{u<t\}), 
		\]
		for a.e.\ $t <t_0.$
		Raising to the power $2/\beta$ the above inequality and combining it with \eqref{eq:holder magic} we reach
		\[
		(\Cap(K_{t_0}))^{-2\alpha/\beta } \left(\int |\d u_{t_0}|^{\beta+1}\d \Per(\{u<t_0\})\right)^{2/\beta}\ge \Cap(K_t)^{\frac{N}{N-2}}\left(\int \frac{1}{|\d u_t|}\d \pert \right)^{-1},
		\]
		for a.e.\ $t <t_0,$ where we have used that $-2\alpha/\beta +(\beta+2)/\beta=\frac{N}{N-2}$. Plugging in \eqref{eq:contr assumption 2} and simplifying, recalling the value of $\alpha,$ we obtain
		\[
		(N-2)^{2\frac{N-1}{N-2}}(\theta\sigma_N)^{\frac{2}{N-2}}\int \frac{1}{|\d u|}\d \pert \ge\frac1t\Cap(K_t)^{\frac{N}{N-2}}=(\Cap(K))^{\frac{N}{N-2}}t^{-\frac{N}{N-2}-1},
		\]
		for a.e.\ $t<t_0.$
		Integrating between $(\tau,t_0)$ we get
		\[
		(N-2)^{2\frac{N-1}{N-2}}(\theta\sigma_N)^{\frac{2}{N-2}}\mea(\{|\d u|>0\}\cap \{u>\tau\}\}) \ge\left(\frac{N-2}{N}\right)(\Cap(K_{\tau})^{\frac{N}{N-2}}-\Cap(K_{t_0})^{\frac{N}{N-2}}).
		\]
		Applying the sharp isocapacitary inequality \eqref{eq:isocap} 
		\[
		(N-2)^{2\frac{N-1}{N-2}}(\theta\sigma_N)^{\frac{2}{N-2}}\mea(K_\tau)\ge \left(\frac{N-2}{N}\right)((N(N-2))^{\frac{N}{N-2}}(\avr(\X)\omega_N)^{\frac2{N-2}}\mea(K_\tau)-\Cap(K_{t_0})^{\frac{N}{N-2}}).
		\]
		Dividing by $\mea(K_\tau)$ on both sides and sending $\tau \to 0^+$, since $\mm(K_\tau)\to +\infty,$ we finally reach
		\[
		(N-2)^{2\frac{N-1}{N-2}}(\theta\sigma_N)^{\frac{2}{N-2}}\ge (N-2)^{2\frac{N-1}{N-2}}(\avr(\X)\sigma_N)^{\frac2{N-2}},
		\]
		which  contradicts the assumption that $\theta<\avr(\X).$
	\end{proof}

	\begin{remark}
		A stronger statement than \eqref{eq:lower bound U} actually holds, meaning that the right-hand side of \eqref{eq:lower bound U} actually coincides  with $\lim_{t \to 0^+} U_\beta(t)$ (from which the inequality follows by monotonicity). The proof of this fact is much more involved than the one for \eqref{eq:lower bound U} and requires careful integral-asymptotic expansion for both $u$ and $|\d u|$. In the smooth setting this was done in \cite{AgostinianiFogagnoloMazzieri20} and later extended to the $\rcd$-setting in \cite[Theorem 3.7.9]{thesis}. Our proof of \eqref{eq:lower bound U} was more direct without using the expression for  $\lim_{t \to 0^+} U_\beta(t)$. This argument  first appeared in \cite{benatti2024minkowski}, where it is used in the  smooth setting to prove an analogous inequality related to the $p$-electrostatic potential.
	\end{remark}
    We are now ready to prove the sharp Willmore inequality stated in the introduction.
	\begin{proof}[Proof of Theorem \ref{thm:sharp willmore}]
    Set $p\coloneqq \beta+1.$
		By Lemma  \ref{lem:w12} $\int |\d u|^{\beta+1} \d\pert$ has a continuous representative in $(0,1)$ that we denote by $R_{|\d u|^{\beta+1}}$. Moreover by Lemma \ref{lem:bv} the function $\int H(u)|\d u|^{\beta} \d \pert$ is in $\BV_\loc(0,1)$ and so has a (pointwise) minimal upper semicontinuous representative that we denote by $R_{H|\d u|^{\beta}}$ (i.e.\ the one obtained by choosing the larger value in the jump points). From Theorem \ref{thm:monotonicity} we know that $U_\beta'\in BV_\loc(0,1)$ and 
		\[
		U_\beta'(t)=\beta t^{-\frac{\beta(N-1)}{N-2}}\left[-R_{H|\d u|^{\beta}} -\frac{N-1}{N-2} \frac1t  R_{|\d u|^{\beta+1}} \right], \quad a.e.\ t\in(0,1).
		\]
		By our choice of $R_{H|\d u|^{\beta}}$ we can then check that the right-hand side must coincide with the maximal lower semicontinuous representative of $U_\beta'$. Hence, again by Theorem \ref{thm:monotonicity},
		\begin{equation}\label{eq:trick monotonicity}
			-R_{H|\d u|^{\beta}}(t) \ge \frac{N-1}{N-2} \frac1t  R_{|\d u|^{\beta+1}}+ \frac1\beta t^{\frac{\beta(N-1)}{N-2}-2}  R(t),  \quad \text{for all $t \in (0,1),$}
		\end{equation}
		where $R(t)\coloneqq  C_{\beta,N} \int_{\{u<t\} }u^2 |\d |\d u^\frac{1}{2-N}|^\frac{\beta}{2}|^2\d \mm\ge0 .$
		On the other hand by Young's inequality for all $s>0$ it holds
		\[
		-R_{H|\d u|^{\beta}}-\frac{s^\beta\beta}{\beta+1} R_{|\d u|^{\beta+1}}\le \frac{1}{s^{\beta^2}}\frac1{\beta+1} W_{\beta+1}, \quad \text{a.e.\ in $(0,1).$}
		\]
		Since the left-hand side is lower semicontinuous, by the maximality of $W_{\beta+1}$ we deduce that the inequality holds pointwise in $(0,1).$ Optimizing in $s$ we obtain
		\begin{equation}\label{eq:holder}
			-R_{H|\d u|^{\beta}}\le (W_{\beta+1})^\frac1{\beta+1} (R_{|\d u|^{\beta+1}})^\frac{\beta}{\beta+1}, \quad \text{pointwise in $(0,1).$}
		\end{equation}
		Combining \eqref{eq:holder} with \eqref{eq:trick monotonicity} we obtain
        \begin{align*}
           (W_{\beta+1})^\frac1{\beta+1}  \ge  \frac{N-1}{N-2} \frac1t  R_{|\d u|^{\beta+1}}^\frac{1}{\beta+1}+ (R_{|\d u|^{\beta+1}})^\frac{-\beta}{\beta+1}\frac1\beta t^{\frac{\beta(N-1)}{N-2}-2}  R(t)
        \end{align*}
		\begin{equation}\label{eq:improved ineq}
			=\frac{N-1}{N-2}  \left(t^{\frac{\beta}{N-2}-1}U_\beta(t)\right)^\frac{1}{\beta+1} +  (R_{|\d u|^{\beta+1}})^\frac{-\beta}{\beta+1}\frac1\beta t^{\frac{\beta(N-1)}{N-2}-2}  R(t)
		\end{equation}
		Since $R(t)\ge 0$ and recalling the lower bound for $U_\beta$  given in \eqref{eq:lower bound U}, we obtain \eqref{eq:willmore ineq intro}.

        Suppose now that equality holds in $\eqref{eq:willmore ineq intro}$ for some $t=t_0$. Then from \eqref{eq:improved ineq} we obtain $R(t_0)=0$ which implies that $ |\d |\d v|^\frac{\beta}{2}|=0$ in $\{u<t_0\}$, where $v\coloneqq u^\frac{1}{2-N}.$ From this we conclude exactly as in \cite[Theorem 7.1]{GigliViolo23}.
	\end{proof}

	In the following stability result we show that almost equality in the Willmore inequality implies that the space is almost a cone.
	\begin{theorem}[Almost rigidity of the Willmore-type inequality]\label{thm:almost rigidity}
For all positive constants $\eps, C,\theta,\rho,\eps_0$, all $N>2$ and $\beta \in (\frac{N-2}{N-1},1]$ there exists $\delta=\delta(\eps,C,\rho,\theta,N,\beta,\eps_0)>0$ such that the following holds. 
		Let $\X$ be an $\rcd(0,N)$ space with $\avr(\X)\ge \theta>0$ and suppose $u$ is an electrostatic potential for a compact set $K\subset \X$. Suppose that $B_\rho(x_0)\subset K$, $\mm(B_1(x_0))\le C,\diam(K)\le C,\Cap(K)\le C$ and $\||\d u|\|_{L^\infty(\{u<t_0\})}\le C$ for some $t_0\in (\eps_0,1).$

        Suppose finally that
        \begin{equation}\label{eq:wil reverse}
       \frac{W_{\beta+1}^\frac{1}{\beta+1}(t_0)}{N-1}  \le (\avr(\X)\sigma_N)^{\frac{\beta}{(\beta+1)(N-2)}}\left(\frac1{t_0}\frac{{\rm Cap}(K)}{N-2}\right)^{\frac{N-2-\beta}{(\beta+1)(N-2)}}   + \delta.
        \end{equation}
		Then 
        \[
        \sfd_{pmGH}((\X,\sfd,\mm,x_0),(\X',\sfd',\mm',x'))<\eps,
        \]
        where $(\X',\sfd',\mm',x')$ is a pointed $\RCD(0,N)$ space which is a truncated cone outside a compact set $K'\subset B_{R_0}(x')$ with $R_0$ depending only on $\rho,N,\theta,\beta,\eps_0$ and $C$ (but not $\eps$).
	\end{theorem}
	\begin{proof}
	    We can assume $\eps<1.$ Combining \eqref{eq:wil reverse} with \eqref{eq:improved ineq} we have
        \begin{equation}
            R(t_0)\le (N-1)\delta (R_{|\d u|^{\beta+1}}(t_0))^\frac{\beta}{\beta+1}\beta t^{2-\frac{\beta(N-1)}{N-2}}\lesssim_{\eps_0,N} \delta  (R_{|\d u|^{\beta+1}}(t_0))^\frac{\beta}{\beta+1},
        \end{equation}
        since $\beta \le 1$ and $t_0\in(\eps_0,1).$ To get an upper bound for $R_{|\d u|^{\beta+1}}$ we write
        \begin{align*}
              R_{|\d u|^{\beta+1}}(t)&=\int |\d u|^{\beta+1}\,\d \pert \le C^\beta \int |\d u|  \,\d \pert =C^\beta \Cap(K),
        \end{align*}
        for a.e.\ $t\in(0,t_0)$. In the last equality we used the fact that $\int |\d u| \pert$ is a.e.\ equal to a constant (see e.g.\ \cite[pag.\ 37]{GigliViolo23}) combined with Proposition \ref{prop:electro cap formula}.  Since $R_{|\d u|^{\beta+1}}$ and rightmost term are both continuous, the final bound holds in fact also at $t=t_0.$ Recalling the expression for $R(t)$ we obtain
        \[
           \int_{\{u<t_0\} }v^{2-N} |\d |\d \sqrt v|^\frac{\beta}{2}|^2\d \mm\le \delta C^\frac{\beta^2}{\beta+1} \Cap(K)^\frac{\beta}{\beta+1}\lesssim_{C,N,\beta}\delta,
        \]
        where $v\coloneqq u^\frac{2}{2-N}.$ Hence the result follows by applying \cite[Theorem 6.7]{GigliViolo23}, once we check that  \sloppy $\||\d \sqrt v|\|_{L^\infty(\{u<t_0\})}\le L$ and $v\ge 1+\eps$ in $\X \setminus B_{R_0}(x_0)$, where $L$ and $R_0$ are constants depending only on $\rho,N,\theta$ and $C$ (but not $\eps$). To this aim we can apply decay bounds for $u$ in \cite[Proposition 5.2]{GigliViolo23} to obtain that
        \begin{equation}\label{eq:decay u}
           c^{-1}\sfd(x_0,x)^{2-N}\le  u(x)\le c \sfd(x_0,x)^{2-N}, \quad \text{for all $x\in \X\setminus B_{r_0}(x_0)$},
        \end{equation}
        where $c\ge 1$ and $r_0$ depend only on $\rho,N,\theta$ and $C.$ We can also assume that $r_0>C$ so that $K\subset B_{r_0}(x_0)$. The second bound in \eqref{eq:decay u} gives  that $v\ge 2\ge 1+\eps$ in $\X \setminus B_{R_0}(x_0)$  choosing $R_0$ sufficiently big depending only on $r_0,c$ and $N.$  To show the bound  on $|\d \sqrt v|$ we use the assumption $\||\d u|\|_{L^\infty(\{u<t_0\})}\le C$  and the fact that $u\ge \min(c^{-1}r_0^{2-N},1)$ in $B_{r_0}(x_0)\setminus K$,  by  the first in \eqref{eq:decay u} and the maximum principle, to obtain
        \[
        \||\d \sqrt v|\|_{L^\infty(\{u<t_0\}\cap B_{r_0}(x_0))}\le\frac{1}{N-2} {C}{\min(c^{-1}r_0^{2-N},1)^{\frac{1-N}{N-2}}}.
        \]
       On the other hand combining the first inequality in \eqref{eq:decay u} with the gradient estimate for harmonic functions (see \cite[Theorem 1.1]{Jiang13}) we have
       \[
          |\d \sqrt v|(x)\le \frac{C(N)}{\sfd(x,\partial K)u(x)^\frac{1}{N-2}}\le  \frac{C(N)c^\frac{1}{N-2}\sfd(x,x_0)}{\sfd(x,\partial K)} \le  \frac{C(N)c^\frac{1}{N-2}\sfd(x,x_0)}{\sfd(x,x_0)-C}\le  \frac{C(N)c^\frac{1}{N-2}r_0}{r_0-C}
       \]
       $\mm$-a.e.\ in $\{u<t_0\}\setminus B_{r_0}(x_0).$  Combining the last two inequalities we obtain the desired upper bound for $\||\d \sqrt v|\|_{L^\infty(\{u<t_0\})}$ and conclude the proof.
	\end{proof}
	Let us stress that the assumption $\||\d u|\|_{L^\infty(\{u<t_0\})}\le C$ in the above statement  is not too restrictive. For example if $K$ is the closure of an open set with smooth boundary in a smooth Riemannian manifold, then $|\d u|$ is bounded even in the whole $\X\setminus K$. In the case of non-smooth spaces see instead the discussion before Proposition 7.5 in  \cite{GigliViolo23}.

\medskip

	\textbf{{Acknowledgments.}}  The second author is funded by the European Union (ERC, ConFine,
101078057).  The present work originated during the writing of the PhD thesis of the second author \cite{thesis}.

	\def\cprime{$'$} \def\cprime{$'$}

\end{document}